\documentclass[12pt]{article}

\usepackage{amsmath, amsfonts, amssymb}

\usepackage[english]{babel}
\usepackage[utf8]{inputenc}
\usepackage{todonotes}
\usepackage{comment}
\usepackage[left=3cm, right=3cm, top=2cm, bottom=3cm, bindingoffset=0cm, marginparwidth=3cm]{geometry}

\usepackage{hyperref}

\usepackage{amsthm}

\usepackage{tikz}
\usetikzlibrary{arrows,calc}
\tikzset{
>=stealth',
}

\newcounter{theorems}
\newtheorem{prop}[theorems]{Proposition}
\newtheorem{thm}[theorems]{Theorem}

\newtheorem{cor}[theorems]{Corollary}

\newtheorem{mytheorem}{Theorem}

\newtheorem{lem}[mytheorem]{Lemma}

\newtheorem*{prob}{Problem}

\newtheorem*{repthm1}{Theorem \ref{thm:NI}}
\newtheorem*{repthm2}{Theorem \ref{thm:instability}}

\theoremstyle{definition}
\newtheorem{defin}[theorems]{Definition}
\theoremstyle{remark}

\newtheorem{rem}[theorems]{Remark}
\newtheorem*{rem*}{Remark}

\makeatletter
\def\blfootnote{\gdef\@thefnmark{}\@footnotetext}
\makeatother

\def\a {\alpha}

\def\s {\sigma }

\def\e {\varepsilon }

\def\be {\begin{equation}}

\def\ee {\end{equation}}

\def\F {\mathcal F}

\def\A {\mathcal A}

\def\B {\mathcal B}

\def\rr {\mathbb R}

\def\zz {\mathbb Z}

\def\nn {\mathbb N}

\def\sign{\operatorname{sgn}}

\def\Diff{\operatorname{Diff}}

\DeclareMathOperator{\dist}{dist}

\def\per{\operatorname{per}}

\title{Locally topologically generic diffeomorphisms with Lyapunov unstable Milnor attractors}

\author{Ivan Shilin\thanks {Moscow State University. This research was supported in part by the Simons Foundation and by the RFBR grant 16-01-00748 a.}}

\date{}

\begin{document}

\maketitle

\begin{abstract}
We prove that for every smooth compact manifold $M$ and any $r\ge1$, whenever there is an open domain in $\Diff^r(M)$ exhibiting a persistent homoclinic tangency related to a basic set with a sectionally dissipative periodic saddle, topologically generic diffeomorphisms in this domain have Lyapunov unstable Milnor attractors. This implies, in particular, that the instability of Milnor attractors is locally topologically generic in $C^1$ if ${\rm dim}\,M\ge~3$ and in $C^2$ if ${\rm dim}\,M = 2.$ Moreover, it follows from the results of C.~Bonatti, L.~J.~D{\'i}az and E.~R.~Pujals that, for a $C^1$ topologically generic diffeomorphism of a closed manifold, either any homoclinic class admits some dominated splitting, or this diffeomorphism  has an unstable Milnor attractor, or the inverse diffeomorphism has an unstable Milnor attractor. The same results hold for statistical and minimal attractors.
\end{abstract}



\bigskip
\section{Introduction}

There exist several nonequivalent definitions of attractors of dynamical systems, in particular, of diffeomorphisms. The general idea is that an attractor is an invariant subset to which most points evolve under the iterates of the system. 

First definitions which appeared in the 1960s required the attractor to coincide with the intersection of the forward images of its dissipative neighborhood. An attracting fixed point --- the simplest example --- and a nontrivial hyperbolic attractor of the Smale-Williams solenoid map perfectly fit these definitions. Here is the simplest example of such definition. 

\begin{defin}[Maximal attractor] Suppose a homeomorphism $F$ has a dissipative domain $U,$ i.e., $\overline{F(U)}\subset U$. Then \emph{a maximal attractor in $U$} is the intersection of the images of this domain under the positive iterates of the system: $A_{max}=\bigcap_{n\in\nn}F^n(U).$ 
\end{defin}

Attractors thus defined are also called \emph{topological} attractors or \emph{trapped} attractors. In this case $U$ is called \emph{a trapping neighborhood}.

\begin{defin} 
An invariant set $K$ of a homeomorphism $F$ is \emph{Lyapunov stable} provided that for any neighborhood $U$ of $K$ there exists another neighborhood $V$ of $K$ such that any future orbit of $F$ starting at $V$ never quits $U$.
\end{defin}

It is easy to see that maximal attractors are always Lyapunov stable.
However, there are very simple dynamical systems for which this definition of attractor is inconvenient when we want to describe a global attracting set. Moreover, C.~Bonatti, M.~Li and D.~Yang proved in \cite{BLY} that if one adds to the definition of attractor the requirement that the attractor must be a chain-transitive set, then there would be $C^1$ locally topologically generic diffeomorphisms with no attractors thus defined.

On the contrary, the following attractors are always defined, at least for diffeomorphisms of compact Riemannian manifolds. 

\begin{defin} [Milnor attractor\footnote{J. Milnor originally called it \emph{the likely limit set}.}, \cite{Milnor}] For a homeomorphism $F$ of a metric measure space, the \emph{Milnor attractor} is the smallest closed set that contains $\omega $-limit sets of almost all orbits. 
\end{defin}
We will denote the Milnor attractor by $A_M$ or $A_M(F).$ In what follows we will always assume that $M$ is a smooth closed Riemannian manifold of dimension at least two with a measure induced by the Riemannian metric.

A simplest example of a Lyapunov unstable Milnor attractor is provided by a diffeomorphism of a circle with a single semistable fixed point, for instance:
$$
x \mapsto x + 0.1 (1-\cos x).
$$
In this example the phase space has no nonempty proper dissipative domains, so we have to either say it has no topological attractor or say that the whole circle is the attractor. Either statement gives little information about the asymptotic behavior of orbits.
The point 0 is the Milnor attractor here, but it is not Lyapunov stable.

The question whether Milnor attractors can be unstable for an open set of diffeomorphisms is still open.
The following result implies though that the Lyapunov stability of attractors is not a topologically generic property and, in fact, instability is rather abundant.

\begin{mytheorem}\label{thm:NI}
Suppose that in a dense subset of an open set $U\subset{\rm Diff}^r(M),\; r\ge 1,$ diffeomorphisms exhibit a homoclinic tangency\footnote{Recall that a homoclinic tangency associated with the periodic saddle $p$ is simply a non-transverse intersection of $W^s(p_1)$ and $W^u(p_2)$, where $p_1, p_2$ belong to the orbit of $p$.} associated with a sectionally dissipative\footnote{The definition of a sectionally dissipative saddle is given in Section~\ref{sec:preliminaries} below (Def.~\ref{def:diss}).} periodic saddle $p$ that continuously depends on the map in $U$. Then a topologically generic diffeomorphism in $U$ has a Lyapunov unstable Milnor attractor. 
\end{mytheorem}

Note that in the hypothesis of Theorem~\ref{thm:NI} we do not assume any continuous dependence of the orbits of tangency on the map.    

Domains satisfying the assumption of Theorem~\ref{thm:NI} were first constructed by S.~Newhouse in his series of works on persistent tangencies in $C^2$ \cite{N70,N74,N79}. He also presented an example of locally dense diffeomorphisms with homoclinic tangencies in $C^1$ in dimension at least three (see Section 8 of \cite{CIME}; this example was later rediscovered by Masayuki Asaoka in \cite{Masa}). Theorem~\ref{thm:NI} may be directly applied to these results, and such application yields, respectively, that on two-dimensional surfaces diffeomorphisms with Lyapunov unstable Milnor attractors are locally topologically generic in $C^2$, and on manifolds of dimension at least three they are locally topologically generic in $C^1$ (see Cor.~\ref{cor:generic2} and \ref{cor:generic3} below). However, as we discuss in Remark~\ref{rem:easy}, one does not actually need Theorem~\ref{thm:NI} to deduce the existence of locally generic diffeomorphisms with unstable attractors. Nevertheless, this theorem is \mbox{necessary} to show that locally generic diffeomorphisms with unstable attractors can be found $C^2$-close to \emph{any} $C^2$-diffeomorphism with a homoclinic tangency related to a sectionally dissipative saddle (Cor.~\ref{cor:approx}).   

C. Bonatti, L. J. D{\'i}az and E.~R.~Pujals stated in \cite{BDP} the following dichotomy for a $C^1$-generic diffeomorphism: for each periodic hyperbolic saddle its homoclinic class either admits a dominated splitting or is contained in the closure of an infinite set of sinks or sources. In Section~\ref{dichotomy} by a rather simple argument based on their work we deduce the following result.
\begin{mytheorem}\label{thm:instability}
Let $M$ be a closed manifold and $F\in\Diff^1(M)$ be a Baire-generic diffeomorphism  of $M$. Then
\begin{itemize}\setlength\itemsep{-0.1em}
\item{either any homoclinic class of $F$ admits some dominated splitting,}
\item{or the Milnor attractor is Lyapunov unstable for $F$ or $F^{-1}.$}
\end{itemize}
\end{mytheorem} 
Note that here we discuss not local, but global topological genericity.

We should note that Theorem~\ref{thm:NI} and its corollaries are also true for diffeomorphisms of a compact manifold with boundary into its interior and even for local diffeomorphisms.\footnote{The definitions of Lyapunov stability and Minor attractor for local diffeomorphisms are the same.} Theorem~\ref{thm:instability} substantially requires the map to be invertible, and therefore this theorem can only be restated for diffeomorphisms of compact manifolds that preserve the boundary. 

Finally, we emphasize that these results on Milnor attractors are also true for any other definition of attractor if this definition implies that
\begin{itemize}\setlength\itemsep{-0.1em}
\item{the attractor exists for any map of a given class,}

\item{the attractor is a closed set,}

\item{the attractor lies in the non-wandering set,}

\item{any hyperbolic periodic sink lies in the attractor.}
\end{itemize}

For example, \emph{statistical} and \emph{minimal} attractors introduced respectively in \cite{AAIS} and \cite{GI} (see also \cite{Il91} and \cite{Il03}) and \emph{the generic limit set} from \cite{Milnor} possess these properties.

\bigskip
\textbf{Acknowledgements.} The author would like to thank Professor Yu.~S.~Ilyashenko for his guidance and S.~S.~Minkov and A.~V.~Okunev for encouraging discussions and many helpful comments.

\section{Preliminaries}\label{sec:preliminaries}
\begin{defin} A property of diffeomorphisms is called \emph{topologically generic} (or \emph{Baire-generic}) if the diffeomorphisms that possess this property form a residual subset\footnote{i.e., a subset that contains a countable intersection of open and dense subsets of the space of diffeomorphisms.} in the corresponding space of diffeomorphisms. A property is called \emph{locally topologically generic} if the diffeomorphisms with this property form a residual subset in some open domain in the space of diffeomorphisms. 
\end{defin}

When we speak about a $C^m$-generic diffeomorphism in $U\subset{\rm Diff}^r(M),\; m\ge r,$ we mean a diffeomorphism from a residual subset in $U\cap {\rm Diff}^m(M)$ (with respect to the subspace topology that comes from $\Diff^m(M)$).  Likewise, we will say that a subset is $C^m$-dense in $U$ if it is dense in $U\cap {\rm Diff}^m(M).$ In what follows, ``generic'' is always assumed to mean ``topologically generic''.

\begin{defin}\label{def:diss}
A hyperbolic periodic saddle $p$ of a diffeomorphism $F$  is called \emph{dissipative} if $|\det \,(dF^{{\rm per}(p)}(p))| < 1$, where $\per(p)$ is the period of $p$. If $|\det\, (dF^{{\rm per}(p)}(p))| > 1$, the saddle is called \emph{area-expanding}. 

The saddle $p$ is called \emph{sectionally dissipative} if it has a unique expanding eigenvalue $\lambda_1: |\lambda_1|>1,$ and for any two eigenvalues $\lambda_i, \lambda_j$ ($i\ne j$) one has $|\lambda_i\cdot\lambda_j|<1.$  
\end{defin}

\begin{rem}\label{rem:volume}
The saddle $p$ is sectionally dissipative iff in some coordinates $dF^{\per(p)}$ contracts two-dimensional euclidean volumes, i.e., its restriction to every (not necessarily invariant) two-dimensional plane is volume contracting (hence the name `sectionally dissipative').
\end{rem}

\begin{defin}
An $F$-invariant set $\Lambda$ is called \emph{locally maximal} if it coincides with the intersection of images of some its neighborhood $U$ under positive and negative iterates of the map: $\Lambda = \bigcap_{n\in\zz}F^n(U(\Lambda)).$ A topologically transitive locally maximal (closed) hyperbolic invariant set is called \emph {a basic set}.  
\end{defin}

Basic sets survive small perturbations of the diffeomorphism, and we will often denote the hyperbolic continuation of a basic set by the same symbol that we use for the original basic set. Periodic points are always dense in basic sets.

\begin{rem}\label{rem:fiber_density}
 Suppose $\Lambda$ is a basic set of saddle type and $p\in\Lambda$ is a periodic saddle. Denote by $O(p)$ the orbit of $p$. Then $W^u(O(p))$ is dense in $W^u(\Lambda).$ Indeed, consider a point $x\in\Lambda$ with a dense forward orbit\footnote{Our set $\Lambda$ is a complete separable metric space, and in this case topological transitivity implies the existence of points with dense forward and backward orbits.}. We may assume that $x$ is close to $p$. Then $W^u(p)\pitchfork W^s(x)\ne\emptyset$ because of the local product structure. But for any point $y$ in this intersection $\dist(F^n(y), F^n(x))\to 0$ as $n\to\infty$. Then the density of the forward orbit of $x$ implies that $W^u_{loc}(\Lambda) \subset \overline{W^u(O(p))}$, and hence $W^u(\Lambda) \subset \overline{W^u(O(p))}$. An analogous statement is true for $W^s(O(p))$ and $W^s(\Lambda)$. 
\end{rem}

Two periodic saddles $p$ and $q$ are called \emph{heteroclinically related} if $W^u(O(p))\pitchfork W^s(O(q))\ne\emptyset$ and $W^s(O(p))\pitchfork W^u(O(q))\ne\emptyset.$ It follows from the $\lambda$-lemma (also known as the inclination lemma) that this is an equivalence relation. The closure of the set of all saddles heteroclinically related to $p$ is called \emph{the homoclinic class} of $p.$ We will denote the homoclinic class of a saddle $p$ by $H(p, F)$, where $F$ stands for the mapping. Whenever we use this notation, we assume that $p$ is a hyperbolic saddle for $F$.

\begin{defin}
We will say that there is a homoclinic tangency associated with a hyperbolic periodic saddle $p$ if $W^u(O(p))$ and $W^s(O(p))$ have a point (and therefore an orbit) of non-transverse intersection.
\end{defin} 

In what follows in dimension greater than two we will consider only homoclinic tangencies that appear for sectionally dissipative saddles.

\begin{defin}\label{def:per_tan}
Suppose that for each diffeomorphism in an open domain $U\subset {\rm Diff}^r(M)$ there is a hyperbolic basic set $\Lambda(F)$ of saddle type that depends continuously on the diffeomorphism and for each $F \in U$ there are two points  $p_1, p_2\in\Lambda (F)$ such that $W^s(p_1)$ has a point of tangency with $W^u(p_2).$ Then we say that there is \emph{a $C^r$-persistent tangency} in $U$ \emph{for a hyperbolic set} $\Lambda(F)$, or that $\Lambda(F)$ \emph{exhibits a $C^r$-persistent tangency}. 
\end{defin}

When we have a diffeomorphism with a basic set such that there is a persistent tangency in the neighborhood of this diffeomorphism associated with the continuation of this basic set, we will informally say that this diffeomorphism and this basic set exhibit a persistent tangency.

The detailed proof of the following proposition can be found in \cite{CIME} (see Lemma 8.4).\footnote{Though the primary focus is on the two-dimensional case there, the proof itself is valid in any dimension.}
\begin{prop}\label{prop:dense}
Suppose that $\Lambda$ is a hyperbolic basic set for a diffeomorphism $F\in \Diff^r(M)$, $p\in\Lambda$ is a periodic saddle, and there are two points $p_1, p_2\in\Lambda$ such that $W^u(p_1)$ and $W^s(p_2)$ have an orbit of tangency. Then a homoclinic tangency associated with (the continuation of) $p$ can be obtained by a $C^r$-small perturbation of $F.$  
\end{prop}

Proposition~\ref{prop:dense} implies that whenever we have a domain $U\subset{\rm Diff^1}(M)$ with a persistent tangency for a hyperbolic basic set $\Lambda(F)$, the diffeomorphisms with a homoclinic tangency associated with a given saddle $p(F)\in\Lambda(F)$ are $C^r$-dense in $U$ for any $r.$ This inspires the definition of a persistent tangency related to a given saddle.

\begin{defin}\label{def:tan_p}
We will say that an open set $U\subset\Diff^r(M)$ exhibits \emph{a $C^r$-persistent tangency associated with the saddle} $p$ of the diffeomorphism $F\in U$ if for any $G\in U$ the continuation of $p$ is defined and diffeomorphisms with a homoclinic tangency related to the continuation of $p$ are dense in $U$ in $C^r$-topology.  
\end{defin}

\begin{defin}
Suppose that $\Lambda$ is an $F$-invariant subset of $M$ and $TM|_\Lambda = E \oplus G$ is a $dF$-invariant splitting of $TM$ over $\Lambda$ such that the dimensions of the fibers of $E$ and $G$ are constant. Then the splitting $E \oplus G$ is called \emph{dominated} if for some $n\in\nn$ for any $x\in\Lambda$ and any $u\in E(x),\; v\in G(x)$ one has
$$ \;\;\frac{\|dF^n(x)u\|}{\|u\|} \le \frac{1}{2}\cdot\frac{\|dF^{n}(x)v\|}{\|v\|}.$$ 
\end{defin}

One may regard the existence of a dominated splitting as a very weak form of hyperbolic-like behavior. If we add to this definition the assumption that $E$ is purely contractive (or $G$ is dilative), we will obtain the definition of partial hyperbolicity; and if we, moreover, assume that $G$ is purely dilative, we will obtain the definition of a hyperbolic set.

\section{A sufficient condition for the Lyapunov instability of the Milnor attractor}

\begin{prop} \label{prop:suff} Suppose that a diffeomorphism $F\in\Diff^1(M)$ satisfies the following conditions:
\begin{itemize}\setlength\itemsep{-0.1em}
\item $F$ has a hyperbolic saddle $p$ whose unstable manifold $W^u(p)$ intersects the basin of attraction of some sink $\gamma$. 
\item $F$ has a sequence of periodic sinks $\gamma_j, \ j\in\nn,$ that accumulate to $p$, i.e.,
    $$
\mbox{\rm dist }(\gamma_j, p) \to 0 \mbox{ as } j \to \infty.
    $$
\end{itemize}
Then $A_M(F)$ is Lyapunov unstable.
\end{prop}

\begin{proof} 

First note that the Milnor attractor $A_M(F)$ lies in the non-wandering set $\Omega(F).$ Indeed, since a wandering point has a neighborhood free of $\omega$-limit points, it can not belong to the attractor, for otherwise one could subtract the neighborhood of this wandering point from the attractor and obtain a smaller closed set that attracts almost every positive orbit, which is a contradiction.

Further note that every sink  $\gamma_j$ belongs to the Milnor attractor since the basin of attraction of $\gamma_j$ contains a neighborhood of this sink and therefore has positive measure.
As these sinks accumulate to the saddle $p$ and the Milnor attractor is, by definition, closed, we get: $p \in A_M(F)$.

Now consider the sink $\gamma$ from the first assumption of the proposition.  The whole basin of this sink, except for the sink itself, consists of wandering points.
The first assumption says that there is a point of $W^u(p)$ in that basin. Since such a point is wandering, it is separated from the attractor. But the preimages of this point under the iterates of $F$ come arbitrarily close to the saddle $p\in A_M(F)$, which implies the Lyapunov instability of the Milnor attractor.
\end{proof}

\begin{rem}
All statements about Milnor attractor which we prove below are reduced to Proposition~\ref{prop:suff}. Note that the proof of Proposition~\ref{prop:suff} uses only the following three properties of the Milnor attractor: the attractor is closed, contains every sink and lies in the non-wandering set. Therefore, analogous statements are true for any other definition of attractor provided that the definition under consideration implies these three properties.
Further note that Proposition~\ref{prop:suff} does not really need the assumption that $F$ is a diffeomorphism: this map may as well be just a local diffeomorphism.
\end{rem}

\begin{rem}
Previously we defined Lyapunov stability in purely topological terms. However, one can consider another definition that is very much alike but not equivalent. Namely, let us call an invariant set $K$ \emph{metrically Lyapunov stable} if for any $\e>0$ there is $\delta>0$ such that any positive semi-orbit that starts $\delta$-close to $K$ never leaves the $\e$-neighborhood of~$K$. Then an analogue of Proposition~\ref{prop:suff} with Lyapunov instability replaced by metric Lyapunov instability is true and the proof does not utilize the closeness of the attractor. Thus, if we confine ourselves with metric Lyapunov stability, we can restate every result of this paper for any other definition of attractor if this definition implies that the attractor contains every sink and lies in the non-wandering set.   
\end{rem}

\section{The Newhouse phenomena and the instability of attractors}

In this section we will prove that every domain with a persistent tangency for a sectionally dissipative saddle has a residual subset of diffeomorphisms with Lyapunov unstable Milnor attractors.

\subsection{The Newhouse phenomena}

In the 70th S. Newhouse proved a number of results on persistent homoclinic tangencies and coexistence of infinitely many sinks or sources.

\begin{thm}[Newhouse, \cite{N70, N74}]\label{thm:N}
For any manifold $M$ of dimension greater than one and for any $r\ge 2$ there is an open set $U\subset\Diff^r(M)$ such that any diffeomorphism $G\in U$ exhibits a persistent tangency for a basic set $\Lambda(G)$ and a topologically generic diffeomorphism in $U$ has infinitely many periodic sinks.
\end{thm}

\begin{thm}[Newhouse, \cite{N79}]\label{thm:N2}
Let $F\in{\rm Diff^2}(M),\; \dim M = 2,$ have a periodic saddle $p$ with a homoclinic tangency. Then arbitrarily close to $F$ there exists an open set $U$ of $C^2$-diffeomorphisms with a persistent tangency for some basic set $\Lambda(G)$ that continuously depends on $G\in U$.
\end{thm}

J. Palis and M. Viana in \cite{PV} generalized the second result for higher dimensions assuming that the saddle $p$ was sectionally dissipative. Although they did not construct a single basic set with persistent tangency, they obtained a persistent heteroclinic tangency for a couple of basic sets and deduced from it the persistent tangency associated with the continuation of some sectionally dissipative saddle. The latter persistent tangency implied a locally generic coexistence of infinitely many sinks. The argument in \cite{PV} generalized the new proof of the Newhouse phenomena presented in the book \cite{PT}.  

\begin{thm}[Palis, Viana, \cite{PV}]\label{thm:ptpv}
Let $F\in{\rm Diff^2}(M)$ have a sectionally dissipative periodic saddle $p$ with a homoclinic tangency. Then arbitrarily close to $F$ there exists an open set $U$ of $C^2$-diffeomorphisms with a persistent tangency associated with the continuation of some sectionally dissipative saddle $p_1$ and topologically generic diffeomorphisms in $U$ have infinitely many sinks.\footnote{Actually, in \cite{PT, PV} only tangencies between invariant manifold of one and the same saddle are considered, but Theorem~\ref{thm:ptpv} is true as stated here; see Proposition~\ref{prop:goodtan3} below.}
\end{thm}

\subsection{Instability of attractors}
This subsection is devoted to the applications of Theorem \ref{thm:NI}.  Let us first repeat this theorem for convenience.

\begin{repthm1}
Suppose that in an open set $U\subset{\rm Diff}^r(M),\; r\ge 1,$ there is a persistent tangency associated with a sectionally dissipative periodic saddle $p$. Then a topologically generic diffeomorphism in $U$ has a Lyapunov unstable Milnor attractor.
\end{repthm1}

\subsubsection{Homoclinic tangencies and instability}

\begin{cor}\label{cor:mger}
Suppose that in an open domain $U\subset\Diff^r(M),\; r\ge 1,$ there is a persistent tangency associated with a basic set $\Lambda(F),\; F\in U,$ with a sectionally dissipative periodic saddle $p(F)$ that continuously depends on $F\in U$. Then a topologically $C^m$-generic ($m\ge r$) diffeomorphism in $U$ has a Lyapunov unstable Milnor attractor. 
\end{cor}
\begin{proof}
Proposition \ref{prop:dense} implies that diffeomorphisms with a tangency associated with the saddle $p(F)$ are $C^m$-dense in $U\cap\Diff^m(M)$. Then we can apply Theorem~\ref{thm:NI} to $U\cap\Diff^m(M)$ and conclude the proof.
\end{proof}

\begin{cor}\label{cor:generic2}
For any smooth compact two-dimensional manifold $M$ there exists a $C^2$-open set $U\subset\Diff^2(M)$ such that any $C^r$-generic ($r\ge2$) diffeomorphism $F\in U$ has a Lyapunov unstable Milnor attractor.
\end{cor}
\begin{proof}
Consider an open set $U$ given by Theorem \ref{thm:N}. Take any $F\in U$ and note that after a small perturbation and, perhaps, time inversion, we may assume that in a $C^2$-neighborhood $V$ of $F$ the basic set $\Lambda$ that exhibits the persistent tangency also has a dissipative saddle. Then Corollary~\ref{cor:mger} applied to $V$ finishes the proof.
\end{proof}

\begin{cor}\label{cor:generic3}
For every smooth compact manifold $M$ of dimension $k\ge 3$ there exists a $C^1$-open domain $U\subset{\rm Diff^1}(M)$ such that any $C^r$-generic ($r\ge1$) diffeomorphism in $U$ has a Lyapunov unstable Milnor attractor.
\end{cor}
\begin{proof}
This is verified by applying Corollary~\ref{cor:mger} to the following theorem\footnote{It is more convenient to give a reference to the work of M.~Asaoka here, though the construction that underlies this result can be found in \cite{CIME}.}.

\begin{thm}[Asaoka, \cite{Masa}]\label{MA}
For any smooth manifold of dimension at least three there exists a $C^{\infty}$-diffeomorphism such that it admits a hyperbolic basic set that contains a sectionally dissipative saddle and exhibits a $C^1$-persistent tangency.
\end{thm}

\end{proof}

\begin{rem}\label{rem:easy}
Actually, one does not need Theorem \ref{thm:NI} to prove either of Corollaries~\ref{cor:generic2} and~\ref{cor:generic3}.

In the two-dimensional case, as it was  observed independently by Yu.~S.~Ilyashenko and A.~Okunev, we may take $F\in{\rm Diff^2}(M)$ as in Theorem~\ref{thm:ptpv}, i.e. with a homoclinic tangency associated with a dissipative saddle $p$, and additionally require that $W^u(O(p))$ intersect the basin of some sink.\footnote{Both conditions can be achieved by isotopically changing a linear mapping with a saddle at the origin, and this procedure may be adapted to any two-dimensional manifold, so the required map exists.} The latter property persists under small perturbations. Therefore we may assume that any diffeomorphism in the domain $U$ given by Theorem~\ref{thm:ptpv} possesses this property. The next step would be to modify the proof of the two-dimensional version of Theorem~\ref{thm:ptpv} given in \cite{PT} in order to show that one may take the saddle $p_1$ that is homoclinically related to the continuation of $p$. This would imply, with the help of the inclination lemma, that for the maps in $U$ the unstable manifold $W^u(p_1(G))$ intersects the basin of the aforementioned sink as well. It is easy to show (and we will show it below) that for a $C^r$-generic diffeomorphism in $U$ sinks accumulate to (the continuation of) $p_1$. Then a $C^r$-generic diffeomorphism in $U$ satisfies both assumptions of Proposition~\ref{prop:suff} with respect to the continuation of $p_1$ and therefore has an unstable Milnor attractor.

Moreover, it turns out peculiarly that the construction presented by M.~Asaoka in \cite{Masa} also ensures that generically both requirements of Proposition~\ref{prop:suff} are satisfied. He utilizes a normally hyperbolic repelling disk such that the restriction of the mapping to this disk is a Plykin map (see the original paper \cite{Ply} by Plykin and also \cite{GUCK} for a slightly modified and somewhat simpler example of such map) with an area expanding saddle that plays a key role in constructing the persistent tangency. One may observe that the stable manifold of this saddle intersects the repelling basin of a source of the Plykin map. After the inversion of time one gets a sectionally dissipative saddle $p$ with $W^u(p)$ intersecting a basin of the sink that was the aforementioned source prior to time inversion. Then a variant of the Newhouse argument (or Baire argument, if you like) yields that sources accumulate to that saddle locally generically.    
\end{rem}

\begin{cor}\label{cor:approx}
Any $C^2$-diffeomorphism $F$ exhibiting a homoclinic tangency associated with a sectionally dissipative periodic saddle $p$ belongs to the closure of a $C^2$-open set $U$ such that a generic diffeomorphism in $U$ has an unstable Milnor attractor.  
\end{cor}
\begin{proof}
Theorem \ref{thm:ptpv} states that arbitrarily close to $F$ there is a domain with a persistent tangency associated with the continuation of some sectionally dissipative saddle $p_1$. Consider a sequence of such domains that approaches $F$ (saddles $p_1$ may be different for different domains). In each domain we apply Theorem~\ref{thm:NI} to the persistent tangency associated with~$p_1$, which yields the genericity of maps with unstable attractors. Thus we can take the union of these domains as $U$.   
\end{proof}

\subsubsection{Infinitely many sinks and instability}

  It must be mentioned that some authors define the \emph{Newhouse phenomenon} as a generic coexistence of infinitely many sinks, not assuming any persistent tangency. Persistent tangencies for sectionally dissipative saddles imply generic coexistence of infinitely many sinks, but  it is not known whether the converse is true.  S.~Crovisier, E.~Pujals and M.~Sambarino have announced (see~\cite[Cor. 4.5]{Crovisier14}) that, at least in $C^1$, locally generic coexistence of infinitely many sinks implies the density of homoclinic tangencies, which seems to be pretty close to having a persistent tangency.

\begin{thm}[Crovisier, Pujals, Sambarino]\label{thm:CPS}
For any open set $V\subset\Diff^1(M)$, $M$ being compact, the following properties are equivalent:
  \begin{itemize}
    \item{Baire-generic diffeomorphisms in $V$ have infinitely many sinks,}
    \item{densely in $V$ there exist diffeomorphisms exhibiting homoclinic tangencies associated with sectionally dissipative periodic points.}
  \end{itemize}
\end{thm}

Since homoclinic tangencies in this theorem appear, in general, for different sectionally dissipative saddles (we do not know whether those are heteroclinically related or not), Theorem~\ref{thm:NI} does not yield the local genericity of diffeomorphisms with unstable Milnor attractors. However, since a $C^1$-diffeomorphism with a homoclinic tangency can be approximated by a $C^2$-diffeomorphism with a tangency related to the continuation of the same saddle, a straightforward application of Corollary~\ref{cor:approx} provides the following statement.

\begin{cor}\label{cor:dense_instability}
Suppose $M$ is a smooth compact manifold and there is an open set $V\subset\Diff^1(M)$ such that a topologically generic diffeomorphism in $V$ has infinitely many sinks. Then densely in $V$ diffeomorphisms have Lyapunov unstable Milnor attractors. 
\end{cor}

To sum up, the positive answer to the following problem seems rather plausible.
\begin{prob}
Is it true that on compact manifolds the local Baire-generic coexistence of infinitely many sinks is always accompanied by the generic instability of Milnor attractors? 
\end{prob}

Recall that the subset of the phase space is called \emph{asymptotically stable} if it is Lyapunov stable and attracts every point in some its neighborhood. The following argument due to A.~Okunev shows that coexistence of infinitely many sinks implies, at least, that the attractor lacks asymptotic stability.  

\begin{thm}[A.~Okunev] If the diffeomorphism $F$ has infinitely many sinks, then its Milnor attractor is not asymptotically stable.
\end{thm}
\begin{proof}
Consider a sink $\gamma$ of $F$. Denote by $B$ the basin of attraction of $O(\gamma)$. Its boundary $\partial{B}$ is a closed invariant set. If it has a nonempty intersection with the attractor, then the attractor is Lyapunov unstable. Indeed, arbitrarily close to a point $a\in \partial{B}\cap A_M$ one can take a point $b\in B$ whose forward orbit inevitably passes through a wandering domain $U\setminus \overline{F(U)}$, \ $U$ being some fixed dissipative neighborhood of $O(\gamma)$.

Suppose now that for every sink of the diffeomorphism $F$ the boundary of the basin of this sink is separated from the attractor. Note then that no point in this boundary is attracted to $A_M$.
Take a sequence $(\gamma_j)_{j\in\nn}$ of sinks that converges to some point $z\in M$ (as always, we assume $M$ to be compact). Since the attractor is closed, $z$ belongs to $A_M$. For every $\gamma_j$ denote by $B_j$ the basin of $O(\gamma_j)$. There is also a sequence of points $b_j\in\partial{B_j}$  that converges to the same point $z$. Indeed, for $\gamma_j$ close to $z$ one can take a segment that connects $z$ with $\gamma_j$ and find a point of $\partial{B_j}$ inside this segment. Thus, we have obtained a sequence $(b_j)$ that converges to $z\in A_M$ but consists of points that are not attracted to $A_M$, which is in contradiction with asymptotic stability. 
\end{proof}


\subsection{Reduction of Theorem \ref{thm:NI} to the capture lemma}\label{sec:thmNI}



\begin{lem}[capture lemma]\label{lem:capture}
Let $F\in{\rm Diff}^r(M), \; r\ge 1,$ have a sectionally dissipative periodic hyperbolic saddle $p = p(F)$ with a homoclinic tangency. Then arbitrary $C^r$-close to $F$ there is a diffeomorphism $G$ for which $W^u(p(G))$ intersects the basin of a periodic sink. 
\end{lem} 

\begin{rem}\label{rem:newhouse_gap}
A two-dimensional version of this lemma can be found in the paper by S.~Newhouse \cite[Lemma 2.2]{N_new} but there is a gap in the proof. Namely, that proof would work only if the eigenvalues $|\lambda|<1<|\sigma|$ of the saddle $p$ satisfied the inequality $|\lambda\sigma^2|<1.$ 
\end{rem} 

\begin{rem}
As J.~C.~Tatjer kindly informed the author, the two-dimensional version of the capture lemma is implied by the main result of the work \cite{TS}(see Thm 5.8 there). Although it seems that the argument of \cite{TS} can be generalized to the general case, below we use a different approach based on the paper \cite{PV} by J.~Palis and M.~Viana. 
\end{rem}

Now we will reduce Theorem~\ref{thm:NI} to the capture lemma, which we will prove in the next section.

Assume as in Theorem~\ref{thm:NI} that there is an open domain $U\subset\Diff^r(M), \; r\ge 1,$ with a persistent tangency associated with a sectionally dissipative periodic saddle $p$ (see Definition~\ref{def:tan_p}) and take a diffeomorphism $F\in U$ with a tangency related to $p(F)$. According to Prop.~1 of~\cite{N74}, when unfolding a homoclinic tangency by a $C^r$-small perturbation, we can create a hyperbolic periodic sink that passes arbitrarily close to the point where the tangency was.\footnote{J.~Palis and M.~Viana prove this for a generic unfolding of a tangency in~\cite{PV} (see Remark~6.1), and we use a similar but a little simpler argument in the proof of the capture lemma, see Section~\ref{sect:gcfeds}.} Since finite segments of orbits depend continuously on the initial point and the mapping, this proposition yields that for any $\delta$ we can create a sink passing $\delta$-close to the hyperbolic continuation of the saddle~$p$. Indeed, before the perturbation the point of tangency $q$ was in the stable manifold of $p(F)$, and therefore there were a neighborhood $V$ of $q$ and an integer $n$ such that $F^n(V)$ lied in the (open) $\delta$-neighborhood of $p(F)$. The same is true for any diffeomorphism $G$ sufficiently close to $F$: $p(G)$ is close to $p(F)$, $G^n(V)$ is close to $F^n(V)$, and therefore $G^n(V)$ is in the $\delta$-neighborhood of $p(G)$. So, if $G$ has a sink in $V$, this sink passes $\delta$-close to $p(G)$. 

At this point we can use a standard argument due to S. Newhouse. Namely, for each $n\in\nn$ denote by $U_n$ the subset of $U$ that consists of diffeomorphisms $F$ with a sink at a distance less than $\frac{1}{n}$ from the saddle $p(F)$. These sets possess the following properties.
\begin{itemize}\setlength\itemsep{-0.1em}
\item{Each $U_n$ is $C^r$-open, because hyperbolic sinks have hyperbolic continuations.}

\item{Each $U_n$ is $C^r$-dense in $U$.

Indeed, we can take any $F\in U$, slightly perturb it to obtain a homoclinic tangency, then make another $C^r$-perturbation to unfold this tangency and create a sink that passes $\frac 1 n$-close to the saddle $p(\cdot)$, which yields a diffeomorphism in~$U_n$.}  
\end{itemize}

Since the sets $U_n$ are $C^r$-open and $C^r$-dense in $U\subset\Diff^r(M)$, the set $R = \bigcap_n U_n$ is a residual subset of $U$. For every $F\in R$ one can find a sink in any neighborhood of $p(F)$, which implies that there is a sequence of sinks that accumulates to this saddle.

Suppose that Lemma~\ref{lem:capture} is true. Then the diffeomorphisms $F$ for which $W^u(p(F))$ intersects a basin of some sink form a $C^r$-dense subset $C$ in $U$. Moreover, this subset is open due to the continuous dependence of local stable and unstable manifolds on the mapping. 

The subset $R\cap C$ is residual in $U$, and any diffeomorphism in this subset satisfies both assumptions of Proposition~\ref{prop:suff}: it has sinks accumulating to a saddle whose unstable manifold intersects the basin of one of the sinks. Therefore, any such diffeomorphism has an unstable Milnor attractor. The proof of Theorem~\ref{thm:NI} modulo Lemma~\ref{lem:capture} is complete.

\section{Proof of the capture lemma}

\subsection{Model example}
In this section we discuss the simplest two-dimensional example of unfolding of a homoclinic tangency. This example shows, without introducing technical difficulties of the general case, how the sink emerges near the point where the tangency was and how a part of the unstable manifold of the saddle is captured by that sink.   

\subsubsection*{Description of the family}
Take two numbers $\lambda, \s\in \rr$ such that $0<\lambda<1<\s$ and $\lambda\s<1$ and consider a plane $\rr^2$ with coordinates $x, y$ and a one-parameter family $\{F_\mu\}_{\mu\in[-\e, \e]}$ of $C^\infty$-diffeomorphisms of that plane such that
\begin{itemize}\setlength\itemsep{-0.1em}
\item{for any $\mu$, the restriction of $F_\mu$ to the rectangle $R_0 = [-2,2]^2$ has the form
\begin{equation}\label{eq:fn}
F_\mu|_{R_0}\colon (x, y)\mapsto(\lambda x, \s y);
\end{equation}
}
\item{there is a small rectangle $R_1$ centered at the point $r = (0, 1)$ and an integer $N\in \nn$ such that for any $\mu$ the restriction $\left.F^N_\mu\right|_{R_1}$ has the form
\begin{equation}\label{eq:fN}
\left.F^N_\mu\right|_{R_1}\colon (x, y)\mapsto \left(y, \;\mu - x + (y-1)^2\right).
\end{equation}
}
\end{itemize} 

Obviously, for any $F_\mu$ there is a fixed dissipative saddle at the origin. We will denote this saddle $p$. The segments $[-2, 2]\times \{0\}$ and $\{0\}\times[-2, 2]$ are local stable and unstable manifolds of $p$ respectively. According to the second condition on the family, $F^N_\mu$ maps the segment $R_1\cap Oy \subset W^u_{loc}(p, F_\mu)$ into an arc $\Gamma_\mu$ of a parabola which shifts along the $y$-axis when we change the parameter, and when $\mu = 0$, there is a homoclinic tangency associated with $p$ at the point $q = (1, 0)$ which is the vertex of $\Gamma_0$ (Figure~\ref{pic}).

\begin{figure}
  \begin{center}

  \begin{tikzpicture}[scale=5]
    \coordinate (x) at (1.5, 0);
    \coordinate (y) at (0, 1.35);
    
    \draw    (y) node[above] {$y$} -- (0,0) coordinate (p) node[below left] {$p$} -- (0,-0.35);
    \filldraw [black] (p) circle (0.3pt);
    \draw    (-0.5, 0) -- (x) node[right] {$x$};
    
    \filldraw [black] (1,0) coordinate (q) node[below, align = center] {$q$\\{\footnotesize $(1,0)$}} circle (0.3pt);
    \filldraw [black] (0,1) coordinate (r) node[left, align = center]  {{\scriptsize $(0, 1)$} $r$} circle (0.3pt);

    \draw[>-<] (-0.25, 0) -- (0.25, 0);
    \draw[<->] (0, -0.25) -- (0, 0.25);
    \node[below] at (0.25, 0) {\footnotesize $W^s_{loc}(p)$};
    \node[left]  at (0, 0.25) {\footnotesize $W^u_{loc}(p)$};
    
    \begin{scope}[yshift = 1cm, scale = 0.25]
     \coordinate (s) at (0, 0);
     \coordinate (a) at (-1, 1);
     \coordinate (b) at (1, 1);
     \coordinate (c) at (1, -1);
     \coordinate (d) at (-1, -1);
    
     \path[draw, dotted] (a) -- (b) -- (c) -- (d) -- cycle; 
     \node[below right] at (-1, 1) {$R_1$};
    \end{scope}

    \begin{scope}[xshift = 1cm, yshift = 0.15cm, scale = 0.2]
     \coordinate (s) at (0, 0);
     \coordinate (a) at (-1, 0.5);
     \coordinate (b) at (1, 0.5);
     \coordinate (c) at (1, -0.5);
     \coordinate (d) at (-1,-0.5);
     \filldraw [black] (s) node[right] {$s_n$} circle (0.5pt);
     \path[draw, thick] (a) -- (b) -- (c) -- (d) -- cycle; 
     \node[right] at (1, 0) {$D_n$};
    \end{scope}

    \filldraw [black] (1,0.15) circle (0.3pt);

    \begin{scope}[xshift = 0.10cm, yshift = 1cm, xscale = 0.05, yscale = 0.3]
     \coordinate (s) at (0, 0);
     \coordinate (a) at (-1, 0.5);
     \coordinate (b) at (1, 0.5);
     \coordinate (c) at (1, -0.5);
     \coordinate (d) at (-1,-0.5);
     \path[draw, thick, color = blue] (a) -- (b) -- (c) -- (d) -- cycle; 
    \end{scope}

    \draw[thick, ->, color = blue] (0.7, 0.15) to [out = 180, in = -90] (0.10, 0.7);
    \node[above right] at (0.3, 0.33) {$F^n_\mu$};

    \begin{scope}[xshift = 0.05cm, yshift = 0.03cm]
    \draw[thick, ->, color = green] (0.25, 1.1) {[rounded corners] to [out = 10, in = 85] (1.05, 0.35)};
    \node[above right] at (0.85, 0.95) {$F^N_\mu$};
    \end{scope}

    \begin{scope}[xshift = 1cm, yshift = 0.185cm, xscale = 0.15, yscale = 0.1]
     
     \coordinate (a) at (-1, 0.5);
     \coordinate (b) at (1, 0.5);
     \coordinate (c) at (1, -0.5);
     \coordinate (d) at (-1,-0.5);
   
     \draw[thick, color = green] (a) parabola[bend pos = 0.5] bend +(0,-0.4) (b) -- (c)  parabola [bend pos = 0.5] bend +(0,-0.4) (d) -- cycle;
     \draw[thick, color = magenta] (-2, 2) parabola[bend pos = 0.5] bend +(0,-1.6) (2, 2) node[above, right, color = black] {$\Gamma_\mu$}; 
    \end{scope}

  \end{tikzpicture}
  \caption{Model example.}\label{pic}
  \end{center}
\end{figure}
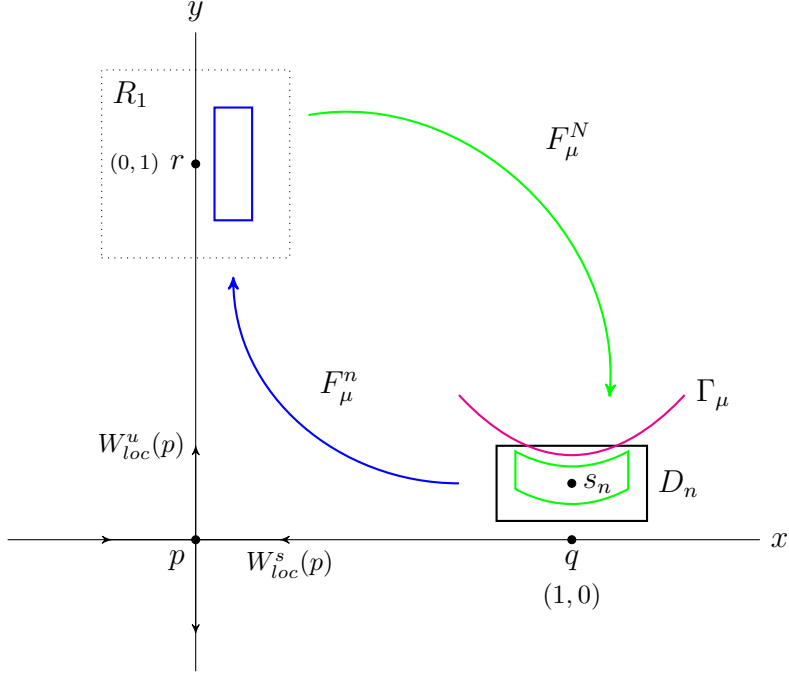

When $\mu$ is small and positive, $\Gamma_\mu$ is shifted upwards and does not intersect the $x$-axis. Take a small rectangle $R$ (with sides parallel to the axes) that lies between the point $q$ and $\Gamma_\mu$. Under the iterates of $F_\mu$ this rectangle will be contracted along $x$-axis and stretched along $y$-axis. If the center of $R$ is at $(1, \s^{-n})$, then the $n$-th image of this rectangle is a thin and long rectangle near the point~$(0, 1)$. The latter rectangle, if it is contained in $R_1$, is mapped by $F^N_\mu$ into a curvilinear figure that lies below $\Gamma_\mu$ and may intersect $R$. The idea is that by adjusting $\mu$ and the size of the rectangle we can assure that  $F^{n+N}_\mu$ maps it into its interior and therefore there is a periodic point inside the rectangle $R$.

\subsubsection*{The sink}
Firstly, we will show how to obtain a sink ``manually''. This will give some motivation to the renormalization technique below.

Suppose that we are looking for a periodic point $s = (x_s, y_s)$ of $F_\mu$ such that $s$ lies in the vicinity of $q$ and goes to the vicinity of the point $r$ after $n$ iterations, and then returns to its original position after another $N$ iterations. If it exists, then $F^n_\mu(x_s, y_s) = (\lambda^nx_s, \s^nx_s)$ lies in $R_1$ and substituting this into (\ref{eq:fN}) yields the analytic expression for $F^{n+N}_\mu(\cdot) = \left(F^N_\mu\circ F^n_\mu\right)(\cdot)$ in a neighborhood of $s$:
\begin{equation}\label{eq:fnN}
F^{n+N}_\mu(x, y) = (\s^ny, \; \mu - \lambda^nx + (\s^ny-1)^2).
\end{equation}
Then we can try to find $s$ by solving the equation
$$(x_s, y_s) = (\s^ny_s, \; \mu - \lambda^nx_s + (\s^ny_s-1)^2).$$

Let us take $x_s = 1$. Then $y_s = \s^{-n}$, and we have to set $\mu = \mu_n = \s^{-n}+\lambda^n$. We have obtained the solution $(x_s,  y_s, \mu) = (1,\; \s^{-n},\; \s^{-n}+\lambda^n)$. Note that $s \to q$ and $\mu_n\to 0$ as $n\to\infty$. Now it is easy to verify that for sufficiently large values of $n$ we have $F^n_{\mu_n}(x_s, y_s)\in R_1$, which implies that $F^{n+N}_{\mu_n}$ restricted to the neighborhood of $s$ is given by expression~(\ref{eq:fnN}). Thus, the point $s$ is indeed a periodic point for the given value of $\mu$.

Using (\ref{eq:fnN}) we can write the Jacoby matrix of $F^{n+N}_{\mu_n}$ in the neighborhood of $s$ explicitly:
\begin{equation}\label{eq:jacoby}
dF^{n+N}_{\mu_n}=
\left(
  \begin{array}{cc}
  0 & \s^n \\
  -\lambda^n & 2\s^n(\s^ny - 1)\\
  \end{array}
\right)
\end{equation}

The matrix of $dF^{n+N}_{\mu_n}(s)$ can be obtained by a substitution $y = y_s = \s^{-n}$. The eigenvalues of this matrix are $\pm i(\lambda\s)^{n/2}$. Since $\lambda\s < 1$, the moduli of eigenvalues are smaller than one and the point $s$ is a periodic sink.

To sum up, for any $n$ we have obtained a parameter value $\mu_n$ such that the map $F^{n+N}_{\mu_n}$ has a sink $s_n$:
$$ s_n = (1, \s^{-n}), \ \mu_n = \lambda^n + \s^{-n}.$$

\subsubsection*{Renormalization}
 Now we can use an argument utilizing a renormalization technique as in \cite[\S 3.4]{PT}. Consider the following $n$-dependent affine change of coordinates:
\begin{equation}\label{eq:newcoords}
 H_n\colon (x, y)\mapsto (X, Y) = (\sigma^n(x - 1),\ \sigma^{2n}(y - \sigma^{-n})).
\end{equation}
The origin of the new $X,Y$-coordinates is at the point $(1, \s^{-n})$ in the $x,y$-coordinates, and the square $K = [-1,1]^2$ in new coordinates corresponds to the following rectangle $D_n$:
$$
D_n = H_n^{-1}(K) = \{ (x,y)\mid |x - 1| \le \s^{-n}, \ |y - \s^{-n}| \le \s^{-2n}\}.
$$
From now on we will consider only large values of $n$ and assume that $F^n_\mu(D_n)\subset R_1.$

Let us calculate the expression for $\F_{n,\mu} = H_n \circ F_{\mu}^{n+N} \circ H_n^{-1}$ at $K$, i.e., the expression for $F_\mu^{n+N}$ at $D_n$ in our new coordinates:
$$
(X,Y) \stackrel{H_n^{-1}}{\mapsto }(x, y) = (\s^{-n}X + 1, \ \s^{-2n}Y + \s^{-n})) \stackrel{F_{\mu}^{n+N}}{\mapsto }
$$
$$
\stackrel{F_{\mu}^{n+N}}{\mapsto }(\s^{-n}Y + 1, \ \s^{-2n}Y^2 - \lambda^n(\s^{-n}X + 1) + \mu ) \stackrel{H_n}{\mapsto }
$$
$$
\stackrel{H_n}{\mapsto } (Y, \ Y^2 - \lambda^n\s^nX + \s^{2n}(\mu  - (\s^{-n} + \lambda^n))) =
\mathcal F_{n,\mu }(X,Y).
$$ 

Consider the following $n$-dependent reparametrization defined in a neighborhood of $\mu_n = \s^{-n} + \lambda^n$:
\begin{equation}\label{eq:reparam}
\nu = \s^{2n}(\mu  - (\s^{-n} + \lambda^n)).
\end{equation}
Then $|\nu| \le 1$ corresponds to $|\mu  - (\s^{-n} + \lambda^n)| \le \s^{-2n}$, and we can suppose that the reparametrization is defined for $\mu$ in a segment of radius $\s^{-2n}$ centered at $\mu_n = \s^{-n} + \lambda^n$ and that $\nu\in [-1, 1]$. 

Let $\mathcal G_{n, \nu(\mu)} := \mathcal F_{n,\mu}.$ Since 
$$\mathcal G_{n, \nu}(X,Y) = (Y, \ Y^2 - (\lambda^n\s^n)X + \nu),$$ 
and $(\lambda\s)^n \to 0$ as $n\to \infty,$ we have
\begin{equation}\label{eq:conv}
\mathcal G_{n,\nu }(X,Y) \rightrightarrows (Y, Y^2 + \nu )=: \mathcal G_\nu(X, Y) \mbox{ for } (X, Y)\in K, \text{ and } |\nu|\le 1 \ (n\to\infty).
\end{equation}
Note that if we take any $k>1$, then for $(X, Y)\in k\cdot K = [-k, k]^2$ and $\nu\in[-k, k]$ the aforementioned reparametrization (\ref{eq:reparam}) and coordinate change (\ref{eq:newcoords}) (more precisely, the inverse ones) will be well-defined for sufficiently large values of $n$ and, moreover, the same uniform convergence will take place. Of course, in this case $H_n^{-1}(k\cdot K)$ will be a rectangle $D_{n,k}$ with the same center as $D_n$ but $k$ times larger, and the original parameter $\mu$ will take on values in the segment 
$\left[\mu_n - k\s^{-2n}, \; \mu_n + k\s^{-2n}\right]$.

The map $\mathcal G_\nu$ (viewed as a map from $\rr^2$ to $\rr^2$) ``forgets'' the $X$-coordinate, maps the square $K$ into an arc of parabola, and is semi-conjugated with the map $Y\mapsto Y^2+\nu$ (from $\rr$ to $\rr$). For $\nu$ small the latter map has exactly two fixed points: a sink at $a_\nu$ close to $\nu$ and a source at $r_\nu$ close to $1 - \nu$, and the orbit of any point $y \in [-r_\nu, r_\nu]$ is attracted to the sink. Thus, for any $r < r_\nu$ the rectangle
$$
B: |Y| \le r, \ |X| \le 1
$$
is attracted to the sink $(a_\nu, a_\nu)$ of $\mathcal G_\nu$, and the distance between the $j$-th image of a point and the sink decreases uniformly on $B$ as $j$ increases.

Since the convergence $\mathcal G_{n,\nu }|_K\to \mathcal G_\nu|_K$ also holds in $C^1$ (in fact, it holds even in $C^\infty$), for sufficiently large $n$ and small $\nu$ the map $\mathcal G_{n,\nu}$ has a sink $s_{n,\nu}$ close to $(a_\nu, a_\nu)$, and it is not difficult to show that $B$ is in its basin of attraction: it is enough to notice that it takes a uniformly limited number of iterations for $B$ to plunge into a dissipative domain where all maps $C^1$-close to $\mathcal G_n$ are contracting in a suitable metric.

\subsubsection*{The capture}

Suppose additionally that
\begin{equation}\label{cond:super_dissipative}
\lambda\s^2 < 1.
\end{equation}
We will call saddles that satisfy such inequality \emph{extremely dissipative}\footnote{\emph{Strongly dissipative} sounds better, but since some authors use this name for sectionally dissipative saddles, we would rather introduce extremely dissipative saddles to prevent any ambiguity.}.

Let us show that for the diffeomorphism $F_\mu$ with any $\mu$ sufficiently close to $\mu_n = \s^{-n} + \lambda^n$ the unstable manifold of the saddle $p$ intersects the attraction basin of the continuation of the sink $H_n^{-1}(s_{n,\nu(\mu_n)})$.

Recall that in the original $x,y$-coordinates we had $r = (0,1) \in W^u(p, F_\mu)$ and $F^N_\mu(r) = (1,\mu)\in W^u(p, F_\mu)$.
Condition~\ref{cond:super_dissipative} implies that $\lambda^n\s^{2n}\to 0$ as $n\to \infty$, and hence we have:
$$H_n(1,\mu ) = (0,\ \s^{2n}\mu - \s^{n} ) = (0,\ \nu + \lambda^n\s^{2n})\to (0, \nu).$$
Obviously, when $n$ is large and $\nu$ is small, $H_n(1,\mu ) \in  B,$ that is, the capture happens. 

\bigskip
Condition (\ref{cond:super_dissipative}) is important for the capture, because when it is violated, the point $(1,\mu)$ may be outside the basin of attraction of the sink $H_n^{-1}(s_{n,\nu(\mu)})$. For example, take $\lambda, \s$ such that $\s>10$ and  $\lambda^2\s^3 = 1$ and consider again $\mu = \mu_n = \s^{-n} + \lambda^n$. Then it is easy to check by a direct calculation using expression (\ref{eq:fnN}) that $F^{2n+N}_{\mu_n}(1, \mu_n) = (\lambda^n+\lambda^{2n}\s^n, \; 2)\notin R_1$ and $F^{2n+N-1}_{\mu_n}(1, \mu_n) = (\lambda^{n-1}+\lambda^{2n-1}\s^n, \; 2\s^{-n})\notin R_1$. Then the image of $(1,\mu_n)$ missed the rectangle $R_1$ on the second ``wind'', and one can add some additional conditions on the family to assure that this orbit never returns to the vicinity of our sink: for example, one can require that, for every map in the family, some neighborhood of the point $(0, 2)$ be attracted to another sink.

 However, due to the complexity of dynamics introduced by homoclinic tangencies and intersections, it is difficult to see whether the global unstable manifold of $p$ intersects the basin of the sink $H_n^{-1}(s_{n,\nu(\mu)})$ when $p$ is not extremely dissipative. In the general case we will use the following trick: a certain perturbation of the map with a homoclinic tangency associated with the saddle $p$ yields a diffeomorphism with an extremely dissipative saddle heteroclinically related to the original one and with a new homoclinic tangency associated with this new saddle.

\subsection{General case for extremely dissipative saddles}
\label{sect:gcfeds}

Let us begin with a general definition of an extremely dissipative saddle in any dimension. Consider a periodic hyperbolic sectionally dissipative saddle $p$ of a map $F$. Let $\s$ be the (unique) expanding eigenvalue of $dF^{{\rm per\,}(p)}(p)$ and $\lambda$ be the restriction of $dF^{{\rm per\,}(p)}(p)$ to the hyperplane $E^s_p$ of contracting eigenvectors. Recall that for any norm on $E^s_p$ the corresponding operator norm of $\lambda$ is defined.   

\begin{defin}
A periodic hyperbolic sectionally dissipative saddle $p$ of a map $F$ is called \emph{extremely dissipative} if there is a norm on $E^s_p$ such that the following inequality holds:
\begin{equation}\label{eq:exdis}
\|\lambda\|\cdot \s^2 < 1.
\end{equation}
\end{defin}

In this subsection we will prove the capture lemma for extremely dissipative saddles.

\medskip

\subsubsection*{Preliminary perturbations}

Suppose $F$ is a $C^r$-smooth ($r\ge 1$) diffeomorphism of $M$ with a homoclinic tangency for an extremely dissipative saddle $p$. Then $F$ can be $C^r$-approximated by a $C^\infty$-diffeomorphism $\hat F$ with the following properties:
\begin{itemize}\setlength\itemsep{-0.1em}
\item $\hat F$ has a non-degenerate homoclinic tangency associated with the con\-tinuation of~$p$;
\item the saddle $p$ is non-resonant for $\hat F$, and all eigenvalues of this saddle are different.
\end{itemize}
Thus, without loss of generality we can assume from the very beginning that $F$ is $C^\infty$-smooth and possesses these properties. To prove the capture lemma for extremely dissipative saddles, it suffices to show that such a diffeomorphism can be $C^\infty$-approximated by a diffeomorphism for which the unstable manifold of the continuation of $p$ intersects a basin of a sink.

\subsubsection*{Linearizing coordinates}

Since $p$ is now non-resonant, by Sternberg's theorem there is a neighborhood of $p$ where $F^{{\rm per\,}(p)}$ can be linearized by a suitable $C^2$-smooth change of coordinates. We will assume for simplicity that $p$ is a fixed saddle, but the argument is almost literally the same for a periodic saddle.

Let us denote by $x, y$ the linearizing coordinates in the neighborhood of $p$: $x\in\rr^{m-1}, \ y\in\rr, \ m = \dim M.$ For the sake of simplicity we identify the points in the domain of the linearizing chart with their images in the codomain. We may assume that in our linearizing coordinates $p$ is the origin, $W^s_{loc}(p)\subset\{y = 0\}$, and $W^u_{loc}(p)\subset\{x = 0\}.$ Then, after a linear change of coordinates that preserves the $y$-axis and the $x$-hyperplane, we can also assume that
\begin{itemize}\setlength\itemsep{-0.1em}
\item[-]{the point $r = (0, 1)\in W^u_{loc}(p)$  belongs to the orbit of tangency and $F^N$ takes $r$ to the point $q = (e, 0)\in W^s_{loc}(p)$, the euclidean norm of $e$ in $\rr^{m-1}$ being smaller than 1;} 
\item[-]{the linear map $\lambda$ 
 contracts the unit ball $B$ of the $x$-plane;}
\item[-]{our coordinates linearize $F$ in a neighborhood $R_0$ of the cylinder $B\times [-1, 1]$.}
\end{itemize}
Let $R_1$ be a small neighborhood of the point $r$. If we write $F^N|_{R_1}$ in the following form:
$$\left.F^N\right|_{R_1}: (x, y) \mapsto (\A(x, y), \B(x,y)),$$ 
then, since there is a quadratic tangency at $q$, the maps $\A\colon \rr^m\to \rr^{m-1}$ and $\B\colon \rr^m\to\rr$ will satisfy:
\begin{equation}\label{eq:cond0}
\A(0,1) = e, \ \B(0, 1) = \partial_y\B(0, 1) = 0, \;\; \partial_{yy}\B(0, 1)\ne 0.
\end{equation}

\subsubsection*{A special family}

Now let us choose a one-parameter $C^\infty$-family  $(F_\mu)$ that passes through our diffeomorphism~$F$. The required perturbation will then be obtained by an arbitrary small shift along the parameter, as we did in the model case.

 When choosing the family we should keep in mind that, first, we want $F^N_\mu$ to have the simplest possible form in the neighborhood $R_1$ of $r$, second, we do not really want the map to change inside $R_0$ when we change the parameter (in particular, we want our saddle to stay at the same place and have the same eigenvalues), and third, we want the tangency to be unfolded non-degenerately. To satisfy these requirements we will consider a family $(F_\mu)_{\mu\in[-\e, \e]}$ such that
\begin{itemize}\setlength\itemsep{-0.1em}
 \item[-]{$F_0 = F$ and any $F_\mu$ coincides with $F$ outside a small neighborhood of the point $z = F^{-1}(q) = F^{N-1}(r)\notin R_0$,}
 \item[-]{when we change the parameter, the $F_\mu$-image of every point in some even smaller neighborhood of $z$ shifts in the  $y$-direction with the same speed equal to 1.}
\end{itemize}
In fact, keeping in mind restrictions (\ref{eq:cond0}) we can make $F^N_\mu$ have the following form in the neighborhood $R_1$ of $r$:
\begin{equation}\label{eq:fNgeneral}
\left.F^N_\mu\right|_{R_1}\colon (x, y)\mapsto \left(e + a\cdot(y-1) + \gamma x + \rho_1(x,y), \;\mu - cx + b(y-1)^2 + \rho_2(x, y)\right),
\end{equation}
where $a\in\rr^{m-1}, \ b\in\rr,$ $c\colon \rr^{m-1}\to \rr$ and $\gamma\colon \rr^{m-1}\to \rr^{m-1}$ are linear maps, and $\rho_{1},\rho_2$ are remainder terms which satisfy the following conditions:  

\begin{equation}\label{rho_conditions}
j^1_{(0,1)}\rho_1 = 0, \ j^1_{(0,1)}\rho_2 = 0,\;\; \partial_{yy}\rho_2(0,1) = 0.
\end{equation}
Here $j^1_t$ stands for the 1-jet of the map at the point $t$. Note that by choosing a proper family and a sufficiently small neighborhood $R_1$ we made the remainder terms independent of $\mu$.
 
Expression (\ref{eq:fNgeneral}) resembles expression~(\ref{eq:fN}) from the model example: the only difference is that several new ``coefficients'' and these two remainder terms have appeared.

As we have already mentioned, in our special family the restriction $F_\mu|_{R_0}$ does not depend on $\mu$:
\begin{equation}\label{eq:fngeneral}
F_\mu|_{R_0}\colon (x, y)\mapsto(\lambda x, \s y).
\end{equation}

\subsubsection*{Renormalization}
Now we can use the renormalization technique as in the model case. We follow \cite{PT, PV}, as before.

Consider the following $n$-dependent affine change of coordinates:
\begin{equation}\label{eq:newcoordsgeneral}
H_n: X = \sigma^n(x - e), \ Y = b\sigma^{2n}(y - \sigma^{-n}).
\end{equation}
In what follows we will consider even $n$ only, so $\s^n$ will always be positive.
The origin of the new coordinates is at the point $(e, \s^{-n})$ in the $x,y$-coordinates and the cube $K = [-k, k]^m, \; k\ge 2,$ in the new coordinates corresponds to the following parallelotope $D_{n, k}$ in the original coordinates:
$$
D_{n,k} = H_n^{-1}(K) = \{ (x,y)\mid |y - \s^{-n}| \le k\s^{-2n}|b^{-1}|, \ |(x - e)_j| \le k\s^{-n}, \; j = 1,\dots, m-1 \}.
$$

As in the model case, we calculate the expression for $\F_{n,\mu} = H_n \circ F_{\mu}^{n+N} \circ H_n^{-1}$ at $K$ assuming that $n$ is sufficiently large:
$$
(X,Y) \stackrel{H_n^{-1}}{\mapsto }(x, y) = (\s^{-n}X + e, \ b^{-1}\s^{-2n}Y + \s^{-n})) \stackrel{F_{\mu}^{n+N}}{\mapsto }
$$
$$
\stackrel{F_{\mu}^{n+N}}{\mapsto }\left(e + a\cdot \frac{\s^{-n}}{b}Y + \gamma(\lambda^n(\s^{-n}X + e)) + \hat\rho_1, \ \mu - c\lambda^n(\s^{-n}X + e) + b(b^{-1}\s^{-n}Y)^2 + \hat{\rho}_2\right ) \stackrel{H_n}{\mapsto }
$$
$$
 \left(ab^{-1}Y + \gamma(\lambda^n(X + \s^{n}e)) + \s^n\hat{\rho}_1, \ Y^2 - cb\s^n\lambda^n(X) + b\s^{2n}(\mu - (\s^{-n} + c\lambda^n(e))) + b\s^{2n}\hat{\rho}_2\right)
$$ 

Here 
$$\hat\rho_i(X, Y) = \rho_i(\lambda^n(\s^{-n}X + e), 1+b^{-1}\s^{-n}Y), \; i = 1,2.$$
 Note that, since the saddle $p$ is dissipative, we have $\|\lambda^n\|\s^n \to 0$ as $n\to \infty$ for any operator norm of $\lambda$. Then, obviously,  $\|\gamma(\lambda^n(\s^{n}e))\|\to 0$,  $\|\gamma\circ\lambda^n\|\to 0$ and $\|cb\s^n\cdot\lambda^n\|\to 0$ as $n\to\infty$. Furthermore, it is easy to check that conditions (\ref{rho_conditions}) yield that
$$\s^n\hat\rho_1|_{K} \to 0, \ \s^{2n}\hat\rho_2|_{K} \to 0 \text{ in } C^1.$$
 
Now let us make, for $n$ sufficiently large, an $n$-dependent reparametrization
\begin{equation}\label{eq:reparamgeneral}
\nu = b\s^{2n}(\mu  - (\s^{-n} + c\circ\lambda^n(e)))
\end{equation}
defined in a closed $k\frac{\s^{-2n}}{|b|}$-neighborhood of $\mu_n = \s^{-n} + c\circ\lambda^n(e)$. Denote this neighborhood by $I_n$. Then $I_n$ corresponds to the interval $[-k, k]$ in the new parameter space.

Take $A = a\cdot b^{-1}\in \rr^{m-1}$ and consider the family of maps
 $$\mathcal G_\nu\colon \rr^{m-1}\times\rr\to\rr^{m-1}\times\rr\colon (X, Y) \mapsto (AY, \; Y^2+\nu).$$ 
We can conclude now that for $\mathcal G_{n, \nu} := \mathcal F_{n,\mu(\nu)}$ restricted to $K$ the following convergence takes place:

\begin{equation}\label{eq:convgeneral}
\mathcal G_{n,\nu }|_K \xrightarrow{C^1} \mathcal G_\nu|_K  \text{ uniformly for } |\nu|\le 1 \text{ as } n\to\infty.
\end{equation}

\subsubsection*{The sink and the capture}

Analogously to the model example, for $\nu$ close to zero the map $\mathcal G_\nu$ has a sink $(Aa_\nu, a_\nu)$, and there is a number $\delta > 0$ independent of $\nu$ such that the cube $B$:
$$
B: |Y| \le \delta, \ |X_i| \le \delta, \ i = 1,\dots, m-1,
$$
is in the basin of this sink. 

Then the convergence (\ref{eq:convgeneral}) implies, for $n$ sufficiently large and $\nu$ sufficiently small, that the map $\mathcal G_{n,\nu}$ has a sink $s_{n,\nu}$ close to $(Aa_\nu, a_\nu)$ and that $B$ is in the basin of attraction of this sink.

Finally, we calculate the $X,Y$-coordinates of the point $F^N_\mu(r)\in W^u(p, F_\mu)$. Since this point has $x,y$-coordinates equal to $(e, \mu)$, we have

$$H_n(e, \mu) = (0, \nu + bc(\s^{2n}\lambda^n(e)))\to (0, \nu) \text { as } n\to\infty.$$
This convergence holds because the saddle is extremely dissipative. Note that though the definition of extreme dissipativity requires condition (\ref{eq:exdis}) to hold in some suitable norm, it nevertheless implies the convergence $\|\s^{2n}\lambda^n\|\to 0$ (as $n\to \infty$) in any norm. 
We see that when $n$ is large and $\nu$ is small, $H_n(e, \mu)$ is in $B$.  We have obtained a diffeomorphism $F_\mu$ for which the unstable manifold of the saddle $p$ intersects the basin of the sink. Recall that the reparametrization was defined for $\mu\in I_n$ and the intervals $I_n$ are close to zero when $n$ is large, therefore such $F_\mu$ can be taken arbitrarily close to $F$, which finishes the proof.  

\medskip

\begin{rem}
We presented the argument for a fixed saddle but, as we have already mentioned, essentially the same argument works for a periodic saddle even if the tangency involves the invariant manifolds of different points of the orbit. The only difference is that for a given $n$ we should consider $F^{n\cdot\per(p)+N}_\mu$ instead of $F^{n+N}_\mu$.
\end{rem}

\begin{rem}\label{rem:genfamily}
The argument of this subsection can be modified in order to show that the capture happens when a quadratic homoclinic tangency associated with an extremely dissipative saddle is unfolded in a generic family of $C^r$-diffeomorphisms, and then one can also get rid of the extreme dissipativity assumption essentially in the same way as it is done in the next subsection. However, we do not need this stronger version of the capture lemma in this paper.
\end{rem}


\subsection{Extreme dissipativity requirement eliminated}

Now we proceed to the case when the saddle $p$ is not extremely dissipative. As in Section~\ref{sect:gcfeds}, we can assume that our map $F$ is $C^\infty$-smooth, the linearization theorem is applicable in the neighborhood of $p$, and the tangency is quadratic. Every perturbation discussed in this section is supposed to be $C^\infty$-small.

\subsubsection*{Finding another saddle}

Imagine that after another small perturbation we find an extremely dissipative saddle $p'$ heteroclinically related to the (continuation of the) original saddle and, more than that, there is a homoclinic tangency associated with $p'$. Then we can apply our capture lemma to extremely dissipative saddle $p'$ and conclude that after another perturbation $W^u(p')$ intersects the basin of a sink. We preserve the original notation $p, p'$ for the continuations of our saddles. 

Since two saddles are heteroclinically related, $\lambda$-lemma implies that $W^u(O(p))$ accumulates on $W^u(O(p'))$. But then $W^u(O(p))$ intersects the basin of the same sink, which proves our lemma in the general case.

The renormalization technique provides a natural candidate for such a saddle $p'$. 
Consider the map $\mathcal G_\nu$ (introduced in the previous subsection) for $\nu = -2$:
$$(X, Y) \mapsto (AY, Y^2 - 2).$$
It has a fixed saddle point $(2A, 2)$ with eigenvalues 4 and 0, the latter being of multiplicity $m-1$. 
Since the maps $\mathcal G_{n, -2}$ approximate the map $\mathcal G_{-2}$ in $C^1$ as $n\to\infty$, for sufficiently large values of $n$ and for $\nu\in [-2-\e, -2+\e]$ the maps $\mathcal G_{n, \nu}$ have extremely dissipative fixed saddles $p_{n, \nu}$ close to $(2A, 2)$. Indeed, for $n$ large, $m-1$ eigenvalues at the saddle $p_{n, \nu}$ are close to zero and the last eigenvalue is close to 4, thus this saddle has to be extremely dissipative. Then the point $\tilde{p}_{n, \mu(\nu)} = H_n^{-1}(p_{n, \nu})$ is an extremely dissipative periodic point of the corresponding diffeomorphism $F_{\mu(\nu)}$. Such a saddle is a natural candidate for the role of the saddle $p'$ because, as the following proposition shows, for some $\nu_0$ close to $-2$ there is a homoclinic tangency associated with the saddle $\tilde{p}_{n, \mu(\nu_0)}$, provided that $n$ is sufficiently large.  


\begin{prop}\label{prop:renormsaddletan}
Let $\e>0$ be fixed. Then, for $n$ sufficiently large, there is $\nu_0\in  {[-2-\e, -2+\e]}$ such that the map $\mathcal G_{n, \nu_0}$ has a homoclinic tangency associated with the saddle~$p_{n, \nu_0}$.
\end{prop}
The proof for the two-dimensional case can be found in \cite{PT}(\S 6.3, Prop. 2), and the generalization of this proof for higher dimensions is straightforward. 
Application of $H_n^{-1}$ yields the corresponding tangency for the saddle $H_n^{-1}(p_{n, \nu_0})$ of $F_{\mu(\nu_0)}$.

\subsubsection*{Heteroclinic relations}

J.~Palis and M.~Viana  study the saddles $\tilde{p}_{n, \mu(\nu)} = H_n^{-1}(p_{n, \nu}), \ \nu\in[-2-\e, -2+\e],$ in~\cite{PV} to construct basic sets with large stable thickness. They prove that, under some assumptions on the original saddle $p$ and on the tangency unfolded to create the saddle $\tilde{p}_{n, \mu}$, this latter saddle is heteroclinically related to the continuation of $p$. In particular, they consider only tangencies between invariant manifolds of one and the same periodic saddle while we allow tangencies between invariant manifolds of different points of a periodic orbit.  

Since we want to prove capture lemma for any sectionally dissipative saddle and any tangency, some additional work has to be done in order to use their results. Our plan is to replace, if necessary, the original saddle, or rather its continuation, with a heteroclinically related saddle, simultaneously obtaining a tangency associated with this new saddle, and then to replace this tangency by another one so that both the new saddle and the new tangency satisfy the aforementioned assumptions of~\cite{PV}. Then the unfolding of this last tangency will yield an extremely dissipative saddle heteroclinically related to the continuations of both previous saddles.  

\bigskip

 At this point we assume that $F\in\Diff^\infty(M)$ has a sectionally dissipative periodic saddle~$p$, that $F^{\per p}$ is linearizable in the neighborhood of $p$, and that the tangency associated with the saddle $p$ is quadratic. We will use the notation introduced in the subsection of Section~\ref{sect:gcfeds} devoted to linearizing coordinates, i.e., we will suppose that these coordinates are defined in a neighborhood $R_0$ of $p$, that $W^u_{loc}({p})$ lies in the $y$-axis and $W^s_{loc}({p})$ lies in the $x$-hyperplane, that there are points $r, q: F^N(r) = q$ in the orbit of tangency, etc. Yet again we do not distinguish the points in $R_0$ and their images under the linearizing chart. When we resort to geometric intuition, we regard the $y$-axis as ``vertical'' and the $x$-plane as ``horizontal''.

The saddle $p$ of $F$ may have several weakest contracting eigenvalues, i.e., contracting eigenvalues with maximal absolute value, but in a generic situation, which we assume to be the case, there is either one real weakest contracting eigenvalue or a pair of conjugated non-real ones. 

Denote by $w$ the number of these weakest contracting eigenvectors.
Let $E^{uw}_{p}$ be the subspace of $T_pM$ spanned by the expanding and the weakest contracting eigenvectors of $p$ and $E^{ss}_{p}$ be the subspace spanned by the rest contracting eigenvectors, which we will call strong. The saddle $p$ has an $(m-w-1)$-dimensional strong stable manifold $W^{ss}(p)$ tangent to $E^{ss}_{p}$. If $M$ is two-dimensional, we take $E^{ss}_p = \{0\}$ and $W^{ss}(p) = \{p\}$. After another small perturbation we can assume that $q\notin W^{ss}(p)$. We will always assume that in what follows.

Linearizing coordinates allow us to define a $dF^{\per(p)}$-invariant splitting $E^{uw} \oplus E^{ss}$ of $T_{R_0}M$: in the linearizing coordinates the fibers $E^{uw}_z, z\in R_0,$ are simply parallel to $E^{uw}_p$, and the fibers $E^{ss}_z$ are parallel to $E^{ss}_p$.

Denote by $\pi$ the projection onto $E^{uw}$ along $E^{ss}$. Consider the map 
$$\Phi \colon E^{uw}_r\to E^{uw}_q, \ \ \Phi = \pi \circ dF^{N}(r)|_{E^{uw}_r}.$$
The linearizing coordinates allow us to identify $E^{uw}_r$ and $E^{uw}_q$ and regard $\Phi$ as a linear map from $\rr^{w+1}$ to itself. Then we can consider the following condition on the tangency at $q$:
\be\label{eq:pvcond}
\det(\pi \circ dF^{N}(r)|_{E^{uw}_r}) > 0.
\ee

In Section 6 of \cite{PV} J.~Palis and M.~Viana consider the case when $q\in W^s_{loc}(p)\cap W^u(p)$ and the saddle $p$ has a unique weakest contracting eigenvalue.   
They prove that if $\Phi$ defined above is an isomorphism, the saddle $\tilde{p}_{n, \mu(\nu)}$, obtained trough the renormalization technique when unfolding the tangency at $q$, has a unique weakest contracting eigenvalue too (for $n$ large). Moreover, if this eigenvalue is positive and there exist transverse homoclinic orbits that involve the same connected components of $W^u(p)\setminus\{p\}$ and $W^s(p)\setminus W^{ss}(p)$ as the former tangency at $q$, then $\tilde{p}_{n, \mu(\nu)}$ is heteroclinically related to the continuation of $p$ when $\nu$ is close to $-2$ and $n$ is sufficiently large. The point is that condition~(\ref{eq:pvcond}) ensures both that $\Phi$ is an isomorphism (which is obvious) and that the weakest contracting eigenvalue of $\tilde{p}_{n, \mu(\nu)}$ is positive, at least when $\nu$ is close to $-2$ and $n$ is large and even (see discussion after (6.10) at p. 244 of \cite{PV}; in the notation of J.~Palis and M.~Viana condition~(\ref{eq:pvcond}) has the form $\det \Delta_{\mu = 0}(r_0) > 0$). For future use let us state as a proposition the exact statement that we need and that has been proved in~\cite{PV} (but not stated as a proposition there).

\begin{prop}[\cite{PV}]
\label{prop:pv}
Suppose that the diffeomorphism $F\in\Diff^\infty(M)$ has a sectionally dissipative saddle $p$ and a quadratic homoclinic tangency between $W^s(p)$ and $W^u(p)$. Suppose, moreover, that the following conditions hold:
\begin{itemize}\setlength\itemsep{-0.1em}
\item[a)]{$p$ has a unique weakest contracting eigenvalue,}
\item[b)]{$F^{\per(p)}$ is linearizable in the neighborhood of $p$,}
\item[c)]{there exists a transverse homoclinic point that belongs to the connected components of $W^u(p)\setminus\{p\}$ and $W^s(p)\setminus W^{ss}(p)$ that are involved in the tangency,}
\item[d)]{condition (\ref{eq:pvcond}) holds for the orbit of tangency (with $N\in\nn$ and points $r\in W^s(p)\cap W^u_{loc}(p)$ and $q \in W^s_{loc}(p)\cap W^u(p)$ as described above).}
\end{itemize}
Then for $\nu$ close to $-2$ and $n$ sufficiently large the saddle $\tilde{p}_{n, \mu(\nu)}$ that appears when generically unfolding the tangency at $q$ in the one-parameter family $F_\mu$ is heteroclinically related to the continuation of $p$.
\end{prop} 

We will always assume $F_\mu$ to be the special family defined in Section~\ref{sect:gcfeds}, but in \cite{PV} they consider an arbitrary one-parameter family that unfolds the tangency generically, i.e., with non-zero speed (see also Remark~\ref{rem:genfamily}). The renormalization scheme works for such a family as well, and the saddles $\tilde{p}_{n, \mu(\nu)}$ can be defined analogously, but since we use our special family only, we do not go into details.

If our diffeomorphism $F$ satisfies all conditions stated in the hypothesis of Proposition~\ref{prop:pv}, then a small perturbation yields an extremely dissipative saddle $\tilde{p}_{n, \mu(\nu)}$ heteroclinically related to the continuation of the original saddle. Due to Proposition~\ref{prop:renormsaddletan}, we can also assume that our new saddle has a homoclinic tangency, so when we unfold this tangency, the unstable manifolds of both the new saddle and the continuation of the original one intersect a basin of a sink.
 
However, for now we suppose that the assumptions of Proposition~\ref{prop:pv} are not satisfied for the tangency that we consider.  Let us show that after an arbitrarily small perturbation one can obtain a diffeomorphism with a homoclinic tangency associated with a saddle $\hat{p}$ heteroclinically related to the continuation of the original saddle $p$ such that for this new tangency Proposition~\ref{prop:pv} is applicable. For that we need to  prove several auxiliary propositions first.

\subsubsection*{Upper and lower tangencies}

Let $F$, as before, have a quadratic homoclinic tangency at $q\in W^s(p)$. Take a connected component $\Gamma$ of $W^u(p)\setminus\{p\}$ that is involved\footnote{\label{footnote:involve}We say that the the connected component $\Gamma$ of $W^u(p)\setminus\{p\}$ is involved in the tangency at $q$ if $q\in F^k(\Gamma)$ for some $k$. Recall that we allow tangencies of invariant manifolds of different points of $O(p)$, so $q$ does not necessarily belong to $\Gamma$, but it is still natural to talk about the component of $W^u(p)\setminus\{p\}$ involved in the tangency. This remark applies to transverse homoclinic intersections as well.} in the tangency at $q$. If the expansive eigenvalue $\sigma$ of $p$ is negative, then both unstable separatrices of $p$ are involved and one can choose any. 
The stable manifold $W^s(p)$ is an injectively immersed open disk of dimension $m-1$. Let us call the ``side'' of $W^s(p)$ that faces $\Gamma$ at $p$ \emph{upper} and the opposite ``side'' \emph{lower}.

To be precise, consider the local stable manifold of $p$ first. This manifold splits a small ball $B_0$ centered at $p$ in two parts; the part that corresponds to $\Gamma$ will be called upper. Given any point $z\in W^s(p)$, take an even number $2l$ such that $F^{2l\cdot\per(p)}(z)\in W^s_{loc}(p)$. Now take a small ball $B\subset M$ centered at $z$. The connected component of $z$ in $W^s(p)\cap B$ splits $B$ in two parts. The upper part is the one that is mapped by $F^{2l\cdot\per(p)}$ inside the upper part of $B_0$ (we assume $B$ to be so small that its $F^{2l\cdot\per(p)}$-image lies inside $B_0$). Note that if $\s$ is negative, $F^{\per(p)}$ swaps the upper and the lower ``sides'' of $W^s(p)$.
 
 Thus, inside a small ball centered at the point $q$ of quadratic tangency the unstable manifold of $O(p)$ approaches $W^s(p)$ either from above or from below. Therefore, we can define the tangencies \emph{from above} and \emph{from below} (or \emph{upper} and \emph{lower} tangencies).  This concept is well-defined for any quadratic tangency associated with a sectionally dissipative saddle. Note that if the expansive eigenvalue of the saddle is negative, then any quadratic tangency related to this saddle becomes a tangency from below when the proper unstable separatrix is chosen. Therefore, in this case we will regard any quadratic tangency as a tangency from below. 

\begin{prop}\label{prop:lowertan}
Suppose that for a diffeomorphism $F\in\Diff^\infty(M)$ a sectionally dissipative periodic saddle $p$ has an upper quadratic tangency at the point $q$ and also a transverse homoclinic orbit that involves the same connected component of $W^u(p)\setminus\{p\}$ as the tangency. Then by an arbitrarily $C^\infty$-small perturbation one can obtain a diffeomorphism $G$ with a lower homoclinic tangency associated with a sectionally dissipative saddle $\hat p$ heteroclinically related to the continuation of $p$. If the original saddle had a unique weakest contracting eigenvalue, then one can take $\hat p$ with the same property. 
\end{prop}
\begin{proof}
We consider the case when the expansive eigenvalue $\sigma$ of $p$ is positive, because otherwise any quadratic tangency is a tangency from below. However, if we did not make this agreement, the argument would stay the same.

As above, we can assume without loss of generality that $F^{\per(p)}$ is linearizable in $R_0\ni p$ and the linearizing coordinates and the points $r, q$ are as described in Section~\ref{sect:gcfeds} ($q$ can be replaced with $F^{2l\cdot\per(p)}(q)$ for some $l > 0$ if necessary). 

Let the points $\tilde{r} = (0, y_{\tilde{r}}) \in W^u_{loc}(p)\cap R_0, y_{\tilde r}\in[0,1/2],$ and $\tilde{q} = (x_{\tilde q}, 0)\in W^s_{loc}(p)\cap R_0$, ${\|x_{\tilde q}\| < 1}, \ {\tilde{q} = F^{\tilde{N}}(\tilde{r})},$ belong to the transverse homoclinic orbit that involves the same connected component of $W^u(p)\setminus \{p\}$ as the tangency at $q$. Without loss of generality we can assume that $\tilde{N}$ is divisible by $\per(p)$ (and so $\tilde{q}\in W^u(p)$): if this is not the case, then it is not difficult to show\footnote{If the unstable separatrix $\Gamma$ of $p$ intersects $W^s(F^k(p))$ transversely, then $\Gamma$ accumulates on $W^u(F^k(p))$. Since $W^u(F^k(p))$ transversely intersects $W^s(F^{2k}(p))$, so does $\Gamma$. Continuing this line of argument, one concludes that $\Gamma$ has a point of transverse intersection with $W^s(F^{k\cdot\per(p)}(p)) = W^s(p)$.} with the help of the $\lambda$-lemma that there is another transverse homoclinic orbit for which this is true.  The existence of such an orbit implies (see, \cite{Katok}, Thm 6.5.5, and also Chapter 3 of \cite{PV}) that there is a basic set $\Lambda\ni p$. 

This basic set $\Lambda$ can be obtained in the following way. For $\delta > 0$ small, take a cylinder 
$$V_\delta = \{(x, y)\in R_0\colon \|x\| \le 1, |y| \le \delta\}.$$
Suppose that the value of $\delta$ is adjusted so that there is $\tilde{n}\in \nn$ such that $F^{\tilde{n}\cdot\per(p)}(V_\delta)$ is a cylinder close to the segment $\{0\}\times[-2y_{\tilde{r}}, 2y_{\tilde{r}}]\subset W^u_{loc}(p)$ and, therefore, $F^{\tilde{n}\cdot\per(p) + \tilde{N}}(V_\delta) \cap V_\delta$ has at least two connected components: one of these components is a cylinder $\Delta_p\ni p$ close to $\{0\}\times[-\delta, \delta]$ and another component $\Delta_{\tilde{q}}$ contains $\tilde{q}$. Denote $F^{\tilde{n}\cdot\per(p) + \tilde{N}}$ by $H$; then $H|_{\Delta_p\cup\Delta_{\tilde{q}}}$ is a horseshoe map (if $\delta$ is small enough) and $\hat\Lambda = \bigcap_{k\in\zz} H^k(\Delta_p\cup\Delta_{\tilde{q}})$ is its maximal invariant set. Then $\Lambda = \bigcup_{k \in \nn} F^k(\hat\Lambda)$ is the required basic set of $F$. We describe this construction to note two important features. First, since $\tilde r$ has a positive $y$-coordinate, this construction implies that in a sufficiently small neighborhood of $p$ every point $\xi\in\Lambda$ lies either in $W^s_{loc}(p)$ or above it relative to the $y$-axis. Second, $F^{-\tilde{N}}(\Delta_{\tilde{q}})$ is contained in a small neighborhood of $\tilde{r}$ when $\delta$ is small.
 
 Thus, it follows from the construction of $\Lambda$ that there is a sequence $(p_j)_{j\in\nn}$ of saddles $p_j\in \Lambda$ such that
\begin{itemize}\setlength\itemsep{-0.1em}
\item[a)]{$\per(p_j) = n_j\cdot\per(p) + \tilde{N}$  (therefore $\per(p_j)$ is divisible by $\per(p)$), }
\item[b)]{$p_j\to p$ and $n_j\to\infty$ as $j\to\infty$,}
\item[c)]{ every $p_j$ lies above $W^s_{loc}(p)$,}
\item[d)]{ there is $r_j\in O(p_j)$ close to $\tilde{r}$ such that $F^{\tilde N}$ maps $r_j$ into a point $q_j$ close to $\tilde{q}$, $F^{l\cdot \per(p)}(q_j)\in R_0$ for $l = 1,\dots, n_j$, and $F^{n_j\cdot \per(p)}(q_j) = r_j$.}
\item[e)]{ all orbits $O(p_j)$ are uniformly far from $F^{-1}(q)$. }
\end{itemize}

First note that properties b) and d) imply that for large $j$ the saddles $p_j$ are sectionally dissipative. Indeed, let us show that in our coordinates they contract the 2-dimensional volume (and use Remark~\ref{rem:volume}).
Since $p$ is sectionally dissipative, we can assume that in the linearizing coordinates on $R_0$ we have $\sigma\|\Lambda\| < 1$, where the norm is generated by the euclidean vector norm. Then in these coordinates $dF^{\per(p)}$ contracts the 2-dimensional volume.\footnote{This can be checked in a straightforward way, e.g., by considering an arbitrary pair of vectors ${(v_i = u_i + s_i \mid u_i \in E^u_p, s_i \in E^s_p)_{i = 1,2}}$ and comparing the Gram determinants of this pair and its image under $dF^{\per(p)}$, expressed in terms of $u_i, s_i, \sigma, \Lambda$.} It follows from b), d) that for $r_j\in O(p_j)$ we have
$$dF^{\per(p_j)}(r_j) = dF^{n_j\cdot \per(p)}(q_j) \circ dF^{\tilde N}(r_j).$$
Here $\tilde N$ is fixed, and so the norm of $dF^{\tilde N}(r_j)$ is uniformly bounded, whereas $dF^{n_j\cdot \per(p)}(q_j)$ coincides with $\left(dF^{\per(p)}(p)\right)^{n_j}$ in our coordinates and hence contracts 2-volume exponentially as $n_j\to\infty$. Thus, for sufficiently large values of $j$, the linear maps $dF^{\per(p_j)}(r_j)$ contract 2-volumes, which proves that $p_j$ are sectionally dissipative.

\bigskip

Furthermore, if $j$ is large enough, then due to the local product structure on $\Lambda$ the local stable manifold $W^s_{loc}(p_j)$ is a small nearly horizontal disk that lies above the $x$-hyperplane and transversally intersects $W^u_{loc}(p)$ at some point $a_j$. Likewise, $W^u_{loc}(p_j)$ is a nearly vertical segment that transversally intersects $W^s_{loc}(p)$ at $b_j$. Denote by $\Gamma_j$ the unstable separatrix of $p_j$ that contains the latter intersection. It is the continuation of $\Gamma_j$ that will be involved in the tangency we are going to obtain. As above, we can consider a ``side'' of $W^s(p_j)$ that faces $\Gamma_j$ at $p_j$. Let us call this ``side'' \emph{positive}, because in this case the word ``upper'' may cause ambiguity: note that the positive ``side'' of $W^s_{loc}(p_j)$ is directed down relative to the $y$-axis (see Figure \ref{lt}).

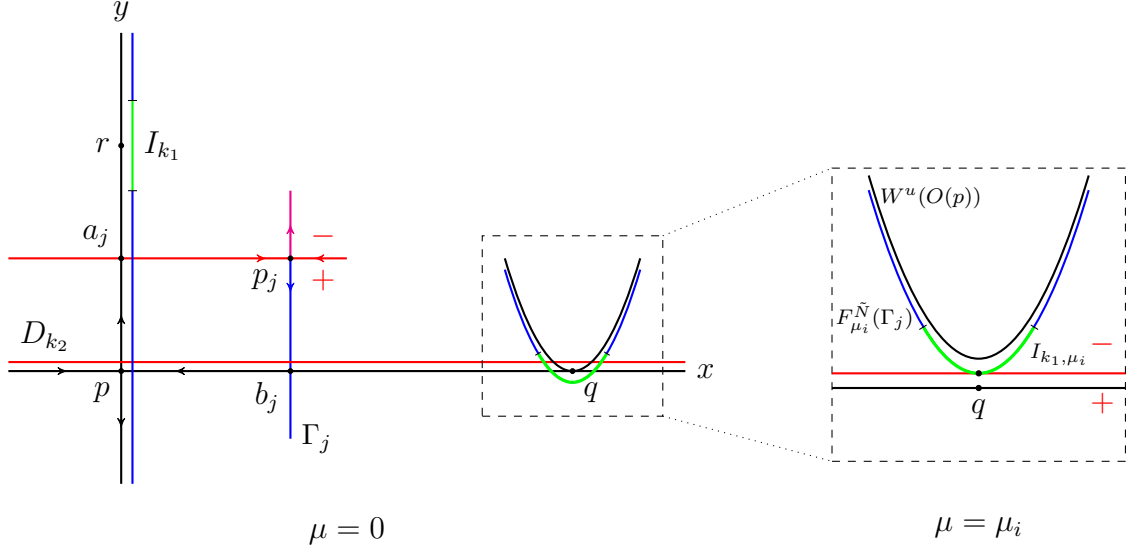
\begin{figure}
  \begin{center}
  \begin{tikzpicture}[scale=3]

    \node [below] at (1, -0.6) {$\mu = 0$};

    \coordinate (x) at (2.5, 0);
    \coordinate (y) at (0, 1.5);
    
    \draw[thick]    (y) node[above] {$y$} -- (0,0) coordinate (p) node[below left] {$p$} -- (0,-0.5);
    \filldraw [black] (p) circle (0.3pt);
    \draw[thick]   (-0.5, 0) -- (x) node[right] {$x$};
    
    \coordinate (q) at (2,0);
    \coordinate (r) at (0,1);

    \draw[->] (-0.25, 0) -- (-0.24, 0);
    \draw[<-] (0.24, 0) -- (0.25, 0);
    \draw[->] (0, -0.24) -- (0, -0.25);
    \draw[<-] (0, 0.25) -- (0, 0.24);

    \begin{scope}[xshift = 0.75cm, yshift = 0.5cm]
    \coordinate (pj) at (0, 0);
    \draw[thick, color = red]   (pj) -- +(-1.25, 0);
    \draw[thick, color = red]   (pj) -- +(0.25, 0) node[below left] {$+$} node[above left] {$-$};
    \draw[color = red, >-<] (-0.15, 0) -- (0.15, 0);
    \draw[color = blue, <-] (0, -0.15) -- (0, 0); 
    \draw[thick, color = blue] (0, -0.8) node[right, black] {\small $\Gamma_j$} -- (0, 0);
    \draw[thick, color = magenta] (0,0) -- (0, 0.3);
    \draw[color = magenta, ->] (0,0) -- (0, 0.15);
    \node[below left] at (pj) {$p_j$} circle (0.3pt);
    \filldraw [color = black] (pj) circle (0.3pt);

    \end{scope}

    \coordinate (aj) at (0, 0.5); \coordinate (bj) at (0.75, 0);
    \node[above left] at (aj) {$a_j$};
    \filldraw [color = black] (aj) circle (0.3pt);
    \node[below left] at (bj) {$b_j$};
    \filldraw [color = black] (bj) circle (0.3pt);

    \draw [thick, color = red] (-0.5, 0.04) node[above right, color = black] {$D_{k_2}$} -- (2.5, 0.04);
    \begin{scope}[xshift = -0.03cm]
    \draw [thick, color = blue] (0.08, -0.5) -- (0.08, 0.8);
    \draw [thick, color = blue] (0.08, 1.2) -- (0.08, 1.5);
    \draw [thick, color = green] (0.08, 0.8) -- (0.08, 1) node[right, black] {$I_{k_1}$} -- (0.08, 1.2);
    \draw (0.06, 0.8) -- (0.1, 0.8); \draw (0.06, 1.2) -- (0.1, 1.2);
    \end{scope}

    \draw[thick] (q) parabola +(0.3, 0.5);
    \draw[thick] (q) parabola +(-0.3, 0.5);
    \begin{scope}[yshift = -0.05cm]
    \draw[thick, blue] (1.7, 0.5) parabola[bend pos = 0.5] bend +(0,-0.5) (2.3, 0.5);
    \draw[very thick, green] (2, 0) parabola +(0.15, 0.125);
    \draw[very thick, green] (2, 0) parabola +(-0.15, 0.125);
    \draw (1.835, 0.12) -- (1.86, 0.135); \draw (2.14, 0.135) -- (2.165, 0.12);
    \end{scope}

    \begin{scope}[xshift = 2cm, yshift = 0.2cm, scale = 0.4]
     \coordinate (a) at (-1, 1);
     \coordinate (b) at (1, 1);
     \coordinate (c) at (1, -1);
     \coordinate (d) at (-1, -1);
    
     \path[draw, dashed] (a) -- (b) -- (c) -- (d) -- cycle; 
    \end{scope}

    \begin{scope}[xshift = 3.8cm, yshift = 0.25cm, scale = 0.65]
     \coordinate (a1) at (-1, 1);
     \coordinate (b1) at (1, 1);
     \coordinate (c1) at (1, -1);
     \coordinate (d1) at (-1, -1);
    
     \path[draw, dashed] (a1) -- (b1) -- (c1) -- (d1) -- cycle; 
     
     \draw[thick] (-1, -0.5) -- (1, -0.5);
     \filldraw[black] (0, -0.5) node[below] {$q$} circle (0.5pt);
     \draw[thick, red] (-1, -0.4) -- (1, -0.4);
     
     \node[left, red] at (1, -0.6) {$+$};
     \node[left, red] at (1, -0.2) {$-$};
     \draw[thick, blue] (-0.75, 0.85) parabola bend (0,-0.4) (0.75, 0.85);
     \draw[thick] (-0.75, 0.95) node[below right] {\scriptsize $W^u(O(p))$} parabola bend (0,-0.3) (0.75, 0.95);
     \draw[very thick, green] (-0.375, -0.088) parabola[bend pos = 0.5] bend (0,-0.4) (0.375, -0.088);
     \draw (-0.405, -0.101) -- (-0.355, -0.07); \draw (0.405, -0.101) -- (0.355, -0.07);
     \node[below left] at (-0.37, 0.16) {\scriptsize $F_{\mu_i}^{\tilde N}(\Gamma_j)$};
     \node at (0.55, -0.27) {\scriptsize $I_{k_1, \mu_i}$};
     \filldraw[black] (0, -0.4) circle (0.5pt);
    \end{scope}

    \node[below] at (3.8, -0.6) {$\mu = \mu_i$};
    \draw[dotted] (b) -- (a1); \draw[dotted] (c) -- (d1);

    \filldraw [black] (2,0) coordinate (q) node[below right, align = center] {$q$} circle (0.3pt);
    \filldraw [black] (0,1) coordinate (r) node[left]  {$r$} circle (0.3pt);

  \end{tikzpicture}
  \caption{Making a lower tangency.}\label{lt}
  \end{center}  
\end{figure}

Since $W^s_{loc}(p_j)$ transversally intersects $W^u_{loc}(p)$ at $a_j$, it follows from the $\lambda$-lemma that there is a sequence of disks $D_k\subset W^s(p_j), \ k\in \nn,$ such that 
\begin{itemize}\setlength\itemsep{-0.1em}
\item[1)]{the disk $D_k$ is a $F^{2k\cdot\per{p_j}}$-preimage of a small neighborhood $\hat{D}_k$ of $a_j$ in $W^s_{loc}(p_j)$ (therefore, all images of these disks under the iterates of $F$ are uniformly far from $F^{-1}(q)$);}
\item[2)]{these disks tend to the disk $D_0 = \{(x, y)\colon y=0, \|x\| \le 1 \}\subset W^s(p)$ as $C^\infty$-immersed disks;}
\item[3)]{the positive ``sides'' of the disks $D_k$ are also directed downwards.}
\end{itemize}
 
Note that 1) implies 3). Indeed, consider an oriented vertical segment $I = Oy \cap R_0$ that goes upwards. It crosses $W^s_{loc}(p_j)$ at $a_j$ from the positive side to the negative side. On one hand, the preimage $F^{-2k\cdot\per(p_j)}(I)$ is also a vertical segment that goes upwards (because $2k\cdot\per(p_j) = 2k\cdot l\cdot\per(p)$ for some $l\in\nn$). On the other hand, a curve that approaches $W^s(p_j)$ from the positive side at $a_j$ is mapped by $F^{-2k\cdot\per(p_j)}$ to a curve that does so at $F^{-2k\cdot\per(p_j)}(a_j)$. Therefore, the positive side of any disk $D_k$ looks downwards.

Likewise, there is a sequence of segments $I_k\subset F^{k\cdot\per(p)}(\Gamma_j), \ k \in \nn,$ that accumulate to some segment $I_0\ni r$ (recall that $r = (0, 1)$) while their preimages are uniformly bounded away from~$F^{-1}(q)$.

Now fix such a saddle $p_j$ with $j$ large and consider a special family $F_\mu$ described in Section~\ref{sect:gcfeds}. Recall that the maps of this family differ from $F_0 = F$ only inside a small neighborhood of $F^{-1}(q)$. If this neighborhood is small enough, $p_j$ is a periodic saddle for every $F_\mu, \ \mu\in[-\e, \e],$ and the disks $D_k$ lie in the stable manifold of this saddle. Likewise, the segments $I_k$ lie in the unstable manifold of the orbit of this saddle. Recall\footnote{See the properties of the special family in Section~\ref{sect:gcfeds}.} that when we change the parameter $\mu$, the intersection of $F^N_{\mu}(I_{k})$ with a sufficiently small neighborhood of $q$ moves upwards or downwards, whereas the disk $D_k$ stands still. Here $k$ is arbitrary. Thus, there is a sequence of parameter values $\mu_i\to 0$ (as $i\to\infty$) such that for each $\mu_i$ there exist $k_1$ and $k_2$ for which $F^N_{\mu_i}(I_{k_1}) =: I_{k_1, \mu_i}$ is tangent to $D_{k_2}$ at some point $\hat{q}$. Moreover, this is a tangency from the negative side (see Figure \ref{lt}). In other words, we have obtained a perturbation of $F$ with a lower tangency associated with $\hat{p} = p_j$.

Now we will prove the second statement of the proposition: if $p$ has a unique weakest contracting eigenvalue, then we can take $\hat p$ with the same property. Let us assume that the image of $E^{ss}_{\tilde q}$ under $dF^{-\tilde N}(\tilde q)$ is transverse to $E^{uw}_{\tilde r}$. This condition can be satisfied by a small perturbation that preserves all relevant properties of $F$, so we suppose that $F$ satisfied it from the very beginning. 

For any $\a > 0$ and $z\in R_0$ denote by $C^{ss}_\a(z)$ the $\a$-cone around $E^{ss}_z$:
$$C^{ss}_\a(z) = \{v\in T_{z}M \mid v = v_{uw} + v_{ss}, \ v_{uw}\in E^{uw}_{z}, \ v_{ss}\in E^{ss}_{z}, \ \|v_{uw}\| \le \a\|v_{ss}\| \}.$$
Consider again the point $q_j\in O(p_j)$ (defined in {\bf\small{d)}} above) and a narrow $\a$-cone $C^{ss}_\a(q_j), \ {\a \ll 1}$. Since all $q_j, \ j\in\nn,$ are close to $\tilde q$, we suppose that the differentials $dF^{-\tilde{N}}(q_j)$ are uniformly close to $dF^{-\tilde{N}}(\tilde q)$. Since $dF^{-\tilde N}(\tilde q)E^{ss}_{\tilde q}$ is transverse to  $E^{uw}_{\tilde r}$, we can also assume that for any $j$ the cone $dF^{-\tilde N}(q_j)(C^{ss}_\a(q_j))$ is contained in $C^{ss}_{\beta}(r_j)$, where $\beta$ is some large number independent of $j$. 

The map $F^{\per(p)}$ is linearized in $R_0$, so we can assume that both this map and its differential at every point $z\in R_0$ coincide with a linear map $L$ for which $E^{ss}$ is a strong stable bundle. Recall that (see {\bf\small{d)}} above) $q_j = F^{-n_j\cdot \per(p)}(r_j)$. The differential $dF^{-n_j\cdot \per(p)}(r_j)$ coincides with $L^{-n_j}$. If $j$ is large, $n_j$ is also large, and if $n_j$ is large enough, then $L^{-n_j}$ maps $C^{ss}_{\beta}(r_j)$ inside $C^{ss}_\a(q_j)$ and expands vectors in $C^{ss}_{\beta}(r_j)$ considerably, so that $C^{ss}_\a(q_j)$ becomes an expansive forward-invariant cone for $dF^{-(n_j\per(p)+\tilde{N})}(q_j)$. Then there is a unique invariant $(m-2)$-dimensional plane inside $C^{ss}_\a(q_j)$ that can only be a span of strong stable eigenvectors of~$q_j$. So, $p_j$ has a unique weakest contracting eigenvalue, at least for $j$ large, and the same is true for its continuation that has a lower tangency related to it.

\end{proof}

\bigskip

\begin{rem}\label{rem:samepoint}
If $q\in W^s(p)\cap W^u(p)$ for $F$, then one can assure that $\hat{q}\in W^u(\hat{p}, G)\cap W^s({\hat p}, G)$.

Indeed, one can argue as follows. The point ${\hat q}$ always lies in $W^s({\hat p}, G)$ by construction. 
If $q\in W^s(p, F)\cap W^u(p, F)$, then $\per(p)$ divides $N$. 
Let $N_1\ge 0$ be the remainder of division of $N$ by $\per(p_j)$. Since ${\per(p)}$ divides $\per(p_j)$ and $N$, it also divides $N_1$.
Consider the segments $I_k\subset F^{k\cdot\per(p)}(\Gamma_j)$, as described in the proof, for $k = \frac{l\cdot\per(p_j) - N_1}{\per(p)}$, where $l\in \nn$. For such $k$ we have that
$$F_\mu^{N}(I_k) \subset F_\mu^{N}\left(W^u(F^{l\cdot\per(p_j) - N_1}(p_j), F_\mu)\right) = W^u(F^{l\cdot\per(p_j) - N_1 + N}(p_j), F_\mu) = W^u(p_j, F_\mu).$$
The last equality holds because ${N_1\equiv N\pmod{\per(p_j)}}$.
If we obtain the new tangency using such $I_k$, this tangency will lie on both invariant manifolds of $\hat{p} = p_j$.
\end{rem}
\bigskip

\begin{rem}
\label{rem:intersections}
We can also assume that the map $G$ in Proposition \ref{prop:lowertan} has transverse homoclinic points that belong to the same connected components of $W^u(\hat{p})\setminus\{\hat p\}$ and $W^s({\hat p})\setminus W^{ss}(\hat p)$ as the newly obtained lower tangency. Indeed, after we fix the saddle $p_j$, whose continuation will play the role of $\hat p$, we can ensure by a small perturbation that the point $a_j\in W^s_{loc}(p_j)\cap{W^u_{loc}(p)}$ does not belong to $W^{ss}(p_j)$. Then for a large  $k_2$ the disk $D_{k_2}\subset W^s(p_j)$ will not intersect $W^{ss}(p_j)$, because $D_{k_2}$ is a distant preimage of a small neighborhood of $a_j$. But this disk definitely will intersect $\Gamma_j$ transversally at some point close to~$b_j$. This transverse intersection is preserved after the perturbation, it involves the same unstable separatrix as the new tangency, and both the new tangency and this intersection lie in the disk $D_{k_2}$ that does not intersect the strong stable manifold of $p_j$.
\end{rem}

\subsubsection*{From lower tangencies to condition (\ref{eq:pvcond})}

\begin{prop}\label{prop:goodtan}
If a diffeomorphism $F\in\Diff^\infty(M)$ has a lower quadratic tangency associated with a sectionally dissipative periodic saddle $p$, then by an arbitrarily $C^\infty$-small perturbation one can obtain a diffeomorphism $G$ with a tangency that is associated with the continuation of $p$ and satisfies condition (\ref{eq:pvcond}).
\end{prop}
\begin{proof}
As above, we suppose that $F^{\per(p)}$ is linearizable in $R_0\ni p$ and the linearizing coordinates and the points $r, q$ are as described in subsection ``Linearizing coordinates'' of Section \ref{sect:gcfeds}: if $q$ is not in $W^s_{loc}(p)$, replace it with $F^{2k\cdot\per(p)}(q)$ for an appropriate $k$. We also assume that $r$ belongs to the unstable separatrix of $p$ that defines the lower tangency: if the expansive eigenvalue $\s$ is positive, this is always the case; otherwise replace $q$ by $F^{\per(p)}(q)$ if necessary.  Thus we can assume that in the neighborhood of $q$ the arc of $W^u(p)$ that is tangent to $W^s(p)$ at $q$ lies in the lower half-space $\{(x, y)\colon y\le 0\}$.

If the tangency at $q$ satisfies condition (\ref{eq:pvcond}), one can take $G = F$, so in what follows we assume that this is not the case.
If $\det(\pi\circ dF^{N}(r)|_{E^{uw}_{r}}) = 0$, this determinant can be made non-zero by a small perturbation that preserves the tangency at $q$, so we can assume that it is negative.

We will obtain a new tangency using the idea of the proof of Thm~1 in \S~3.2 of~\cite{PT}. Namely, consider a special\footnote{In fact, this argument can be adapted for any one-parameter family that unfolds the tangency at $q$ generically.} one-parameter family $(F_\mu)_{\mu\in[-\e,\e]}$ described in Section~\ref{sect:gcfeds}. Let $D\ni q$ be a small disk in $W^s_{loc}(p)$  such that $D_0 := F^{-N}(D)$ lies in  $R_0$ and is $\delta$-distant from $\partial{R_0}$, while the boundary of $D_0$ is $\delta$-distant from $W^u_{loc}(p)$, $\delta$ being some small positive number. Let $I_0 = \{0\}\times[1-\e_1, 1+\e_1]$ be a small neighborhood of the point $r$ in $W^u_{loc}(p)$.  Then for small positive values of $\mu$ the arc $F_{\mu}^N(I_0)$ has two points $z_1(\mu), z_2(\mu) \in D$ of transverse intersection with $W^s_{loc}(p)$. Denote by $\Gamma_\mu$ the segment of this arc that lies above $W^s_{loc}(p)$. Note that $D_{\mu} := F^{-N}_{\mu}(D)$ transversally intersects $W^u_{loc}(p)$ at points $w_i(\mu) = F_{\mu}^{-N}(z_i(\mu)), \ i = 1, 2,$ close to $r$. If $\mu$ is sufficiently small, we can assume that $\partial{D_\mu}$ is $\delta$-distant from $W^u_{loc}(p)$ and $D_\mu$ itself is $\delta$-distant from $\partial{R_0}$. 

The $\lambda$-lemma implies that for sufficiently large even $n\in\nn$ the arc $\Gamma_{n, \mu} := F_{\mu}^{n\cdot\per(p)}(\Gamma_\mu)$ has points of transverse intersection with $D_{\mu}$. Fix some $n$ such that this holds and, moreover, $\|\lambda^n(e)\| < \delta/2$ (recall that $e$ is the $x$-coordinate of $q$). Let $\hat{z}_i(\mu) = F_\mu^{n\cdot\per(p)}(z_i(\mu)), \: {i = 1,2},$ and let $\gamma_i(\mu)$ be a connected component of $\hat{z}_i(\mu)$ in $\Gamma_{n, \mu}\cap R_0$. Note that $\gamma_1(\mu)$ coincides with $\gamma_2(\mu)$ when $\Gamma_{n, \mu}$ is contained in $R_0$.

For every point in $\Gamma_\mu$ its $x$-coordinate is close to that of the point $q = (e, 0)$, therefore the $x$-coordinate of any point that belongs to $\gamma_1(\mu)\cup \gamma_2(\mu)$ is close to $\lambda^n(e)$. In other words, $\gamma_1(\mu)\cup \gamma_2(\mu)$ (viewed as a set) is very close to $W^u_{loc}(p)$. This property is preserved when $\mu$ decreases to zero and the curve $\Gamma_{n, \mu}$ shrinks towards the point $(\lambda^n(e), 0) = F^n(q)$. At the same time $D_{\mu}$ stays in a small neighborhood of the point $r$ and far from $\partial{R_0}\cup W^s_{loc}(p)$, and $\partial{D_\mu}$ stays $\delta$-far from $W^u_{loc}(p)$. Thus when we decrease $\mu$, the curves $\gamma_1(\mu)$ and $\gamma_2(\mu)$ cannot intersect $\partial{D_\mu}$ because the latter is $\delta$-far from $W^u_{loc}(p)$, and $D_\mu$ cannot intersect $\partial(\gamma_1(\mu)\cup \gamma_2(\mu))\subset \partial{R_0}\cup W^s_{loc}(p)$. However, for sufficiently small values of $\mu$, $\gamma_1(\mu)$ coincides with $\gamma_2(\mu)$ and does not intersect $D_\mu$. Therefore, for some $\mu_n>0$ there is a point $r_0$ of tangency between one of the curves $\gamma_i(\mu_n)$ and $D_{\mu_n}$. Note that this construction allows $\mu_n$ to be taken arbitrarily close to zero and $r_0$ to be arbitrarily close to $r$. We will take $G = F_{\mu_n}$. As always, we can assume without loss of generality, that the tangency at $r_0$ for $G$ is quadratic.

\begin{figure}
  \begin{center}
  \begin{tikzpicture}[scale=0.52]

    \node [below] at (7, -6) {$\mu = \mu_n$};

    \coordinate (x) at (17.5, 0);
    \coordinate (y) at (0, 10);
    
    \draw[thick, red]   (-3, 0) -- (0,0) coordinate (p) node[below left, black] {$p$} -- (x);

    \draw[->, red] (-2.3, 0) -- (-2.2, 0);
    \draw[<-, red] (2.2, 0) -- (2.3, 0);

    \draw[->, blue] (0, -2.4) -- (0, -2.5);
    \draw[<-, blue] (-0.06, 2.5) -- (-0.055, 2.4);    
    
    \coordinate (q)    at (11.2,0);
    \coordinate (hatq) at (12.4,0);
    \coordinate (hatr) at (-0.07,4);
    \coordinate (r_0)  at (1.2, 3);
    \coordinate (q_0)  at (11.2, 0.6);

    \coordinate (right0) at (0.35, 0);
    \coordinate (left1)  at (1, 0);
    \coordinate (right1) at (1.4, 0);
    \coordinate (left2)  at (3.5, 0);
    \coordinate (right2) at (4.3, 0);
    \coordinate (left3)  at (10.3, 0);

    \draw[thick, ->] plot[smooth, tension = 1] (10.5, 1.5) {[rounded corners] to [out = 175, in = -40] (2, 4)};
    \node[above right] at (5.5, 2.) {$F_{\mu}^n$};

    \draw[thick, ->] plot[smooth, tension = 1] (2., 6.) {[rounded corners] to [out = 51, in = 87] (12, 2.5)};
    \node at (8.3, 7.7) {$F_{\mu}^N$};

    \draw[thick, blue] plot [smooth, tension = 0.65] coordinates{(0,-5) (0,0) (0,5) (0.9,7.6) (3,9.7) (6,10.81) (10,10.35) (13, 8.2) (15, 4) (15, -1.5) (14.2, -3.7) (12.8, -3.5) (12, 0)}; 

    \draw[thick, blue] plot [smooth, tension = 0.65] coordinates{(0.2,0) (0.2,4.9) (1.1,7.5) (3.2,9.6) (6,10.61) (10,10.15) (12.88, 8.08) (14.82, 3.92) (14.8, -1.3) (14.2, -3.4) (13, -3.4) (12.5, -1.5) (12.35, -0.5) (12.4, 0) (12.5, -0.22) };

    \draw[thick, blue] plot [smooth, tension = 0.65] coordinates{(0.35,0) (0.36,4.83) (1.27,7.43) (3.4,9.5) (6,10.41) (10,9.95)  (12.76, 7.9)  (14.64, 3.87) (14.6, -1.2) (14.1, -3.1) (13.1, -3.1) (12.6, -0.7)  (12.5, -0.22) (12.4, 0)};

    \draw[thick, red] plot [smooth, tension = 1.7] coordinates{(2, 2.95) (-1,4) (2, 5.05)};

    \draw[thick, blue] plot [smooth, tension = 1.5] coordinates{(left1) (r_0) (right1)};

    \draw[thick, blue] plot [smooth, tension = 1.7] coordinates{(left2) (3.9,1.3) (right2)};

    \draw[thick, blue] plot [smooth, tension = 1.5] coordinates{(10.3, 0) (11.24,0.62) (12, 0)};

    \draw[dotted, blue] plot [smooth, tension = 1] coordinates{(right0) (0.53, -4) (0.58, 1.99) (0.85, -3) (left1)};
    \draw[dotted, blue] plot [smooth, tension = 1.4] coordinates{(right1) (1.7, -2) (2.3, 1) (2.99, -1.5) (left2)};
    \draw[dotted, blue] plot [smooth, tension = 1.1] coordinates{(right2) (5.2, -1.) (6.9, 0.4) (8.9, -0.6) (left3)};

    \filldraw [black] (p) circle (2pt);
    \node[above left] at (-1, 4) {$D_{\mu}$};
    \node[above right] at (1.3, 1.2) {$\Gamma_{n,\mu}$};
    \filldraw [color = black] (q)    circle (2pt) node[below]       {$q$};
    \filldraw [color = black] (hatq) circle (2pt) node[above right] {$\hat q$};
    \filldraw [color = black] (hatr) circle (2pt) node[left]        {$\hat r$};
    \filldraw [color = black] (r_0)  circle (2pt) node[above] {$r_0$};
    \filldraw [color = black] (q_0)  circle (2pt) node[above]       {$q_0$};

    \filldraw [color = black] (left1)   circle (1.5pt);
    \filldraw [color = black] (right1)   circle (1.5pt);
    \filldraw [color = black] (2, 2.95)  circle (1.5pt);
    \filldraw [color = black] (2, 5.05)  circle (1.5pt);

  \end{tikzpicture}
  \caption{The new tangency.}\label{pic3}
  \end{center}  
\end{figure}

Let $q_0 = G^{-n\cdot\per(p)}(r_0)$, $\hat{r} = G^{-N}(q_0)$ and $\hat{q} = G^N(r_0) = G^{n\cdot\per(p) + 2N}(\hat{r})$. The point $\hat{r}$ lies in $W^u_{loc}(p)$ and is close to $r$. The points $q_0$ and $\hat{q}$ are close to $q$, and $\hat{q}\in W^s_{loc}(p)$ (see Fig.~\ref{pic3}). 
We are going to prove that for the tangency at $\hat{q}$ condition (\ref{eq:pvcond}) holds. The point $\hat{r}$ plays the same role for $\hat{q}$ that $r$ played for $q$, so condition (\ref{eq:pvcond}) takes the following form:

\begin{equation}\label{eq:pvc2}
\Delta := \det(\pi\circ dG^{n\cdot\per(p) + 2N}(\hat{r})|_{E^{uw}_{\hat r}}) > 0.
\end{equation} 

Since $dG^{n\cdot\per(p) + 2N}(\hat{r}) = dG^{N}(r_0) \circ dG^{n\cdot\per(p)}(q_0) \circ dG^{N}(\hat{r})$ we need to prove that
$$ \sign\det\left(\left.\pi\circ dG^{N}(r_0) \circ dG^{n\cdot\per(p)}(q_0) \circ dG^{N}(\hat{r})\right|_{E^{uw}_{\hat r}}\right) = 1.$$

\medskip
Recall that $dG^{n\cdot\per(p)}(q_0) = dF_{\mu_n}^{n\cdot\per(p)}(q_0) = L^{n}$. Let us introduce the following shorthand notation:
$$\Xi = dG^{N}(r_0), \ \ \Theta = dG^{N}(\hat{r}).$$
Note that both $\Xi$ and $\Theta$ are close to $dF_{0}^{N}(r)$ since $r_0$ and $\hat{r}$ are close to $r$ and $\mu_n$ is close to zero.

\begin{rem*}
Let us deal with the two-dimensional case first. If $\dim M = 2$, (\ref{eq:pvc2}) is reduced to $\det(\Xi \circ L^n \circ \Theta) > 0$. Since $\Xi \approx dF^N_0(r) \approx \Theta$, we have $\sign\det(\Xi) = \sign\det(\Theta)$. Recall that $n$ is even and therefore $\det(L^n) > 0$. Combining these two observations, we finish the proof. In the general case we implement the same idea.
\end{rem*}

Let us assume that $dF^{N}(r)(E^{uw}_{r})$ is transversal to $E^{ss}_q$: this is a generic property compatible with tangency at $q$. Since $\Theta$ is close to $dF^{N}(r)$, we suppose that $\Theta(E^{uw}_{\hat r})$ is transversal to $E^{ss}_{q_0}$. If $n$ is large, $L^n\circ\Theta(E^{uw}_{\hat r})$ is a plane $E_0$ very close to $E^{uw}_{r_0}$. If it is close enough, we have
$$
\sign\Delta = \sign\det\left(\left.\pi \circ \Xi \circ L^n \circ \Theta\right|_{E^{uw}_{\hat r}}\right) = 
\sign\det\left(\pi \circ \Xi|_{E^{uw}_{r_0}}\right) \cdot 
\sign\det\left(\pi \circ L^n \circ \Theta|_{E^{uw}_{\hat r}}\right).
$$
Since $\Xi$ is close to $dF^{N}(r)$, we have
$$
\sign\det\left(\pi \circ \Xi|_{E^{uw}_{r_0}}\right) = \sign\det\left(\left.\pi \circ dF^{N}(r)\right|_{E^{uw}_r}\right) = -1.
$$
Furthermore, since the bundle $E^{uw}$ is invariant for $L^n$, we can write
$$
\sign\det\left(\pi \circ L^n \circ \Theta|_{E^{uw}_{\hat r}}\right) =
\sign\det\left(\pi \circ L^n|_{E^{uw}_{q_0}} \circ \pi \circ \Theta|_{E^{uw}_{\hat r}}\right) =
$$
$$
=\sign\det\left(\pi \circ L^n|_{E^{uw}_{q_0}}\right) \cdot \sign\det\left(\pi \circ \Theta|_{E^{uw}_{\hat r}}\right).
$$
Finally, $\sign\det\left(\pi \circ L^n|_{E^{uw}_{q_0}}\right) = 1$ since $n$ is even, and 
$$
\sign\det\left(\pi \circ \Theta|_{E^{uw}_{\hat r}}\right) = \sign\det\left(\left.\pi \circ dF^{N}(r)\right|_{E^{uw}_r}\right) = -1.
$$  

Therefore, $\sign\Delta = -1\cdot 1\cdot (-1) = 1$, which concludes the proof.
\end{proof}

\medskip

\begin{rem}\label{rem:goodtan}
Note that for the new point of tangency we have $$\hat{q}\in {W^u(G^{2N+n\cdot\per(p)}(p))\cap W^s(p)} = W^u(G^{2N}(p))\cap W^s(p)$$ (we denote the continuation of the original saddle $p$ of $F$ by the same symbol). If $q\in W^u(p, F)\cap W^s(p, F)$, then $N$ is divisible by $\per(p)$ and, therefore, $\hat{q}\in W^u(p, G)\cap W^s(p, G)$.
\end{rem}

\subsubsection*{From tangencies to transverse intersections}

\begin{prop}\label{prop:trans}
Suppose that a diffeomorphism $F$ has a homoclinic tangency associated with a sectionally dissipative saddle $p$. Then either there are transverse homoclinic intersections that involve\footnote{See footnote~\ref{footnote:involve}.} the same connected component of $W^u(p)\setminus\{p\}$ as the tangency, or such intersections can be obtained by a small perturbation together with a new homoclinic tangency.
\end{prop}
\begin{proof}
As always, we can assume that $F$ is $C^\infty$-smooth, $F^{\per(p)}$ is linearizable in the neighborhood of $p$, and the tangency is quadratic. Denote the point of tangency by $q$. 

If the tangency at $q$ is a tangency from below, we can argue as in the proof of Proposition~\ref{prop:goodtan} and obtain a new tangency and transverse homoclinic intersections as required no matter if such transverse orbits existed prior to perturbation or not.

Suppose now that we have a tangency from above at $q$. Denote by $\Gamma$ the connected component of $W^u(p)\setminus\{p\}$ involved in the tangency. Suppose that for the diffeomorphism $F$ there are no transverse intersections between $F^k(\Gamma)$ and $W^s(p)$ for any $k\in\zz$.
If there is an orbit of tangency different from that of $q$, we can perturb one of these orbits into an orbit of transverse intersection while preserving the second orbit unchanged. 

Thus, we need to consider only the case when there are no intersections between $\bigcup_k F^k(\Gamma)$ and $W^s(p)$ other than those that belong to $O(q)$. In this case the argument is analogous to the first part of the proof of Proposition~\ref{prop:goodtan}, and even simpler.

Namely, we can consider the linearizing neighborhood $R_0$ and assume that $q\in W^s_{loc}(p)\subset R_0$ and $r = F^{-N}(q)\in W^u_{loc}(p)\subset R_0$, just like in the proof of Proposition~\ref{prop:goodtan}. Consider again a small disk $D\colon q\in D\subset W^s_{loc}(p)$, and its $F^N$-preimage $D_0$ contained in a neighborhood of $r\in\Gamma$. We assume that $\partial{D_0}$ is $\delta$-far from $W^u_{loc}(p)$ and $D_0$ itself is $\delta$-far from $\partial{R_0}\cup W^s_{loc}(p)$.

Consider a special one-parameter family $(F_\mu)$ that unfolds the tangency at $q$. Let $D_{\mu} := F_{\mu}^{-N}(D)$. Since we have a tangency from above when $\mu = 0$, for small (in absolute value) $\mu < 0$ there will be two transverse intersections between $F_{\mu}^N(W^u_{loc}(p))$ and $W^s_{loc}(p)$ at the points $z_1(\mu), z_2(\mu)$ near $q$. Take some small $\mu_0<0$ and denote by $\Gamma_{\mu_0}$ a small arc in $W^u(O(p))$ that starts at $z_1(\mu_0)$ and goes up. For a sufficiently big even integer $n$ there is a transverse intersection between $F^{n\cdot\per(p)}_{\mu_0}(\Gamma_{\mu_0})$ and $D_{\mu_0}$ near~$r$.

Let $R$ be the intersection of $R_0$ with the upper half-space and $\gamma(\mu), \ \mu < 0,$ be a connected component of the point $\hat{z}_1 = F^{n\cdot\per(p)}_{\mu}(z_1)$ in $W^u(O(p))\cap R$. Then for large $n$ and small negative $\mu$ the component $\gamma(\mu)$ continuously depends on $\mu$.  Define the curve $\gamma(0)$ by continuity. If $n$ is sufficiently large, then for any $\mu\in[\mu_0, 0]$ both endpoints of $\gamma(\mu)$ lie far from $D_\mu$, while $\gamma(\mu)$ itself is at least $\delta/2$ close to $W^u_{loc}(p)$. At the same time we can assume that $\partial{D_\mu}$ is $\delta$-far from $W^u_{loc}(p)$. Recall that for $\mu = 0$ we have no intersections between $\gamma(\mu)$ and $D_\mu$, but for $\mu = \mu_0$ there is a transverse intersection. Since for $\mu\in [\mu_0, 0]$ the curve $\gamma(\mu)$ and the disk $D_\mu$ can intersect by interior points only, there is some value of $\mu$ when a point of tangency appears. Thus we have obtained both transversal intersections and a tangency as required. 
 
\end{proof}

\begin{rem}\label{rem:trans}
If for $F$ the point of tangency $q$ is not in $W^{ss}(p)$, then the same argument yields that we either have  a transversal intersection that involves the same connected components of $W^u(p)\setminus\{p\}$ and $W^s(p)\setminus W^{ss}(p)$ or can obtain such an intersection together with a new tangency that involves the same connected components (to be precise, their continuations). It suffices to notice that the new transverse or tangential intersections are constructed near the original orbit of tangency with respect to the metric on $W^s(p)$.

Furthermore, if $q\in W^u(p)\cap W^s(p)$, then this intersection also is in $W^u(p)\cap W^s(p)$ (since $\per(p)$ divides~$N$).
\end{rem}

\subsubsection*{Tangencies for invariant manifolds of the same saddle}

\begin{prop}\label{prop:goodtan2}
If a diffeomorphism $F$ has a lower quadratic tangency between $W^s(p)$ and $W^u(F^N(p))$, where $p$ is a sectionally dissipative periodic saddle, then by an arbitrarily $C^\infty$-small perturbation one can obtain a diffeomorphism $G$ with a tangency between the stable and the unstable manifolds of the continuation of $p$.
\end{prop}
\begin{proof}
Consider a diffeomorphism $G$ and a point $\hat{q}$ as in the proof of Proposition~\ref{prop:goodtan}.
Note that for $G$ we have transversal intersections between $W^u(p)$ and $W^s(G^{-N}(p))$ (at the points $w_i(\mu_n)$ mentioned in the proof). The $\lambda$-lemma implies that $W^u(p)$ accumulates on $W^u(G^{-N}(p))$ and, therefore, transversally intersects also $G^{-N}(W^s(G^{-N}(p))) = W^s(G^{-2N}(p))$. Arguing inductively, we obtain that for any $k\in\nn$ the unstable manifold $W^u(p)$ transversally intersects $W^s(G^{-kN}(p))$ and accumulates on $W^u(G^{-kN}(p))$ . Take $k = 3\per(p) - 2 > 0$. This yields that $W^u(p)$ transversally intersects $W^s(G^{2N}(p))$ and accumulates on $W^u(G^{2N}(p))$. But $W^u(G^{2N}(p))$ is tangent to $W^s(p)$ at $\hat{q}$. Therefore, we can obtain a tangency between $W^u(p)$ and $W^s(p)$ by another small perturbation arguing as in the proof of Proposition~\ref{prop:lowertan}. 
\end{proof}

\begin{prop}\label{prop:goodtan3}
Suppose that a diffeomorphism $F$ has a sectionally dissipative periodic saddle $p$ and there is a quadratic tangency between $W^s(p)$ and $W^u(F^N(p))$ at the point $q$.  Then by an arbitrarily $C^\infty$-small perturbation one can obtain a diffeomorphism $G$ with a tangency between the stable and the unstable manifolds of the continuation of $p$.
\end{prop}
\begin{proof}
Since the case of a lower tangency was already considered in the previous proposition, we assume that there is a tangency from above at $q$.  
Without loss of generality we can also suppose that $q\notin W^{ss}(p)$ and the saddle $p$ has transverse homoclinic orbits that involve the same connected components of $W^u(p)\setminus\{p\}$ and $W^s(p)\setminus W^{ss}(p)$ as the orbit of tangency (see Proposition~\ref{prop:trans} and Remark~\ref{rem:trans}).  

Since the tangency at $q$ is a tangency from above, Proposition \ref{prop:lowertan} yields that we can switch to another saddle heteroclinically related to $p$ and obtain, by a small perturbation of $F$, a lower tangency associated with the continuation of this new saddle. Let us denote the perturbed map by $\hat{F}$ and the new saddle by $\hat{p}$ and preserve the notation $p$ for the continuation of the original saddle. Note that we can take $\hat{p}$ such that $\per(\hat{p}) = l\cdot\per(p)$ (for some $l\in\mathbb N$), $W^u(p, \hat{F})\pitchfork W^s(\hat{p}, \hat{F})\ne\emptyset$, and $W^s(p, \hat{F})\pitchfork W^u(\hat{p}, \hat{F})\ne\emptyset$, as in the proof of Proposition~\ref{prop:lowertan}. 

Applying Proposition~\ref{prop:goodtan2} to $\hat{F}$, we obtain a diffeomorphism $\hat{G}$ such that there is a tangency between the stable and the unstable manifolds of the continuation of $\hat{p}$. Again, let us preserve the notation $p, \hat{p}$ for the continuations of our saddles.

Since $W^u(p, \hat{G})\pitchfork W^s(\hat{p}, \hat{G})\ne\emptyset$ and $\per(p) \mid \per(\hat{p})$, we have that $W^u(p, \hat{G})$ accumulates on $W^u(\hat{p}, \hat{G})$. Analogously, since $W^s(p, \hat{G})\pitchfork W^u(\hat{p}, \hat{G})\ne\emptyset$, we have that $W^s(p, \hat{G})$ accumulates on $W^s(\hat{p}, \hat{G})$. Then, arguing as in the proof of Proposition~\ref{prop:lowertan}, we can make a small perturbation and obtain a diffeomorphism $G$ with a tangency between the stable and the unstable manifolds of the continuation of $p$, which finishes the proof. 

\end{proof}

\subsubsection*{End of the proof}

Now when we proved the auxiliary propositions, we can get back to the proof of the capture lemma.
Recall that we assumed $F\in \Diff^\infty(M)$ to have a sectionally dissipative saddle $p$ with a quadratic homoclinic tangency between $W^s(p)$ and $W^u(F^{N}(p))$ at the point $q$. Our goal is to match all assumptions of Proposition~\ref{prop:pv} and obtain an extremely dissipative saddle heteroclinically related to the continuation of $p$ and having an orbit of tangency. However, we assume that Proposition~\ref{prop:pv} is not applicable to $F$ yet. 

We are going to perform a series of perturbations, each of them arbitrarily $C^\infty$-small, and obtain in the end a diffeomorphism, a saddle,  and a tangency for which Proposition~\ref{prop:pv} is applicable. In order to simplify the notation, we will repeatedly replace $F$ with the perturbed maps. At each step we also replace our saddle $p$ either by its hyperbolic continuation or by another saddle heteroclinically related to this continuation, and we obtain a new point of tangency, which we continue to denote by $q$ though it is a different point associated with the new saddle $p$.

{\bf 1.} First, by Proposition~\ref{prop:goodtan3} we can obtain, by a small perturbation, a new tangency between the stable and the unstable manifolds of the continuation of $p$. Thus, we replace the diffeomorphism $F$ by the perturbed map and preserve the original notation for the continuation of $p$ and for the new point of tangency. If we already had a tangency between $W^s(p)$ and $W^u(p)$, this step is redundant.

{\bf 2.} Now we assume that $F$ has a quadratic tangency between $W^u(p)$ and $W^s(p)$ at $q$, the saddle $p$ being, as always, sectionally dissipative.
In Section 5 of \cite{PV} J.~Palis and M.~Viana prove that in this case a small perturbation yields another sectionally dissipative saddle $\tilde{p}$ that has a homoclinic tangency associated with it, has a unique weakest contracting eigenvalue and, moreover, is heteroclinically related to the original one. 

\medskip
There might be a small gap in the argument of J.~Palis and M.~Viana, but it is easy to fix it.
In order to prove this result they first impose some genericity condition\footnote{They suppose that the map $\Phi$ (defined before condition~(\ref{eq:pvcond})) is an isomorphism.} on $F$ and then consider a one-parameter family $(F_\mu), \ F_0 = F,$ that generically unfolds the tangency. Their argument suggests, however, that for $\mu > 0$ the tangency is unfolded with creation of transverse intersections and that there is also a sequence $(\mu_j)_{j\in \nn}, \ \mu_j > 0, \ \mu_j\to 0$ as $j\to\infty$, such that the maps $F_{\mu_j}$ have new tangencies associated with the continuation of $p$. But it may happen that this is not the case: it is possible that the new tangencies appear only for negative $\mu$.
However, this is always the case for a lower tangency, as one can see from the first part of the proof of Proposition \ref{prop:goodtan}.
Thus, one can use Proposition \ref{prop:lowertan} (with Remark~\ref{rem:samepoint}) to switch to a heteroclinically related saddle with a lower tangency first and then argue as J.~Palis and M.~Viana do.
\medskip

As before, we replace $F$ by the perturbed map and $p$ by $\tilde{p}$ without changing the notation.

{\bf 3.} Now we assume that the saddle $p$ has a unique weakest contracting eigenvalue and there is\footnote{We might need another perturbation to move the tangency off $W^{ss}(p)$ and make it non-degenerate.} a quadratic homoclinic tangency at $q\in W^u(p)\cap (W^s(p)\setminus W^{ss}(p))$. Using Proposition~\ref{prop:trans} with Remark~\ref{rem:trans}, we can also assume that there is a transverse homoclinic intersection between the same connected components of $W^u(p)\setminus\{p\}$ and $W^s(p)\setminus W^{ss}(p)$. 

If the tangency at $q$ is a tangency from above, we can apply Proposition \ref{prop:lowertan} (with Remark~\ref{rem:samepoint}) to obtain a new tangency from below between $W^u(p)$ and $W^s(p)\setminus W^{ss}(p)$. In this case we replace the map, the saddle, and the tangency by the new ones again. By Remark~\ref{rem:intersections}, we can assume that there still are transversal homoclinic orbits that involve the same connected components of $W^u(p)\setminus\{p\}$ and $W^s(p)\setminus W^{ss}(p)$ as the tangency. If we have accidentally lost linearizability of $F^{\per(p)}$ in the neighborhood of $p$ (actually, we could not), we can restore it by a small perturbation that preserves all relevant properties of our map.

{\bf 4.} Thus, now we assume that the tangency at $q$ is a tangency from below. If this tangency does not satisfy condition~(\ref{eq:pvcond}), then apply Proposition~\ref{prop:goodtan} with Remark~\ref{rem:goodtan} to obtain, after a small perturbation, a new tangency that satisfies condition~(\ref{eq:pvcond}). We can assume that the linearizability is preserved. Replace the map and the tangency by the new ones. The new tangency belongs to the same connected components of $W^u(p)\setminus\{p\}$ and $W^s(p)\setminus W^{ss}(p)$ as before (recall that in the proof of Prop.~\ref{prop:goodtan} $\hat{q}$ is close to $q$), therefore there are still transversal homoclinic intersections as in condition c) of Proposition~\ref{prop:pv}. Then Proposition~\ref{prop:pv} can be applied to our new $F$.

{\bf 5.} Proposition~\ref{prop:pv} yields that after another perturbation we finally obtain  an extremely dissipative saddle heteroclinically related to the continuation of the saddle $p$. Using Proposition~\ref{prop:renormsaddletan}, we can also suppose that for this new map there is a homoclinic tangency associated with this extremely dissipative saddle.
\medskip

At each step of our argument we could replace saddle $p$ either by its continuation or by a heteroclinically related saddle only, so this final extremely dissipative saddle is heteroclinically related to the continuation of the original saddle $p$ that existed prior to any perturbation.
We can apply the capture lemma to the tangency associated with the extremely dissipative saddle and conclude, as above, that after a proper perturbation the unstable manifolds of the continuations of both the extremely dissipative saddle and the original one intersect a basin of a sink. Now the proof of the capture lemma is complete.


\section{Dominated splitting or instability}\label{dichotomy}

\subsection{Statement and plan of the proof}

C. Bonatti, L. J. D{\'i}az and E. R. Pujals prove in \cite[Cor. 0.3]{BDP}  the following dichotomy for a $C^1$-generic diffeomorphism of a closed manifold: for each periodic hyperbolic saddle its homoclinic class either admits a dominated splitting or is contained  in the closure of an infinite set of sinks or sources. In this section by a relatively simple argument also based on \cite{BDP} we will deduce Theorem~\ref{thm:instability} from a technical statement that underlies the result we have just quoted. Let us recall the statement of Theorem~\ref{thm:instability} for convenience.

\begin{repthm2}
For a Baire-generic diffeomorphism $F\in\Diff^1(M)$, $M$ being a closed manifold, either any homoclinic class of $F$ admits some dominated splitting, or the Milnor attractor is Lyapunov unstable for $F$ or $F^{-1}.$
\end{repthm2}

Note that the two cases considered in Theorem~\ref{thm:instability} are not mutually exclusive. Indeed, one can take locally generic diffeomorphisms with unstable attractors from the previous sections, multiply them by a strong contraction, and thus obtain a dominated splitting whereas the attractors will still be unstable. What Theorem \ref{thm:instability} really says is that the absence of a dominated splitting over some homoclinic class generically implies instability of the Milnor attractor (perhaps, for the inverse map).

We will need the following fact that, though not stated explicitly, is proved in \cite{BDP}.

\begin{thm}[\cite{BDP}, Lemma 1.9 + Lemma 1.10 + Prop. 2.1]\label{BDP_p1}
Suppose that $p$ is a periodic hyperbolic saddle of the diffeomorphism $F\in\Diff^1(M)$ and the homoclinic class $H(p,F)$ does not admit a dominated splitting. Then, for any sufficiently small $\e>0,$ in any neighbourhood of $p$ one can find a periodic saddle $q$ with the following properties: 
\begin{itemize}\setlength\itemsep{-0.1em}
\item[--]{ $q$ is heteroclinically related to $p$,}
\item[--]{ there are linear maps $A_j$ $\e$-close to the differentials $dF(F^{j-1}(q)), \ j = 1,\dots, {\rm per \,}(q),$  such that the composition $A = A_{{\rm per \,}(q)}\circ\dots\circ A_1$ is either a contraction or a dilation.\footnote{In this context a contraction (dilation) is meant to be a linear map with moduli of eigenvalues less than~1 (greater than~1), and we do not assume that the norm of this map is necessarily less than~1 (resp., greater than~1).  Then, since $\e$ is arbitrary, we don't actually need to specify which norm we use to define the $\e$-perturbation of $dF$. By default, we will assume the operator norm that corresponds to the vector norm provided by the Riemannian structure.}}
\end{itemize}
\end{thm}


\begin{rem}  \label{rem:sinks}
Actually, it follows from Remarks 5.5 and 5.6 in \cite{BDP} that if $H(p,F)$ contains a dissipative saddle $p_1$ heteroclinically related to $p$ then one can take $q$ such that the perturbation of the differentials along $O(q)$ yields a contraction. Respectively, if there is an area-expanding $p_2$, one may assume that $A$ is a dilation. 
\end{rem}

Recall that according to Franks' lemma \cite[Lemma 1.1]{Franks} an $\e$-perturbation of the differential $dF$ over a finite subset $B\subset M$ can be realized by a diffeomorphism $G$ that is $10\e$-close to $F$ in $C^1$ and coincides with $F$ on $B$ and outside some neighborhood of $B$ that may be chosen to be arbitrary small. Therefore Theorem~\ref{BDP_p1} combined with Franks' lemma may be viewed as a means of creating sinks near the (continuation of the) saddle $p$.

In order to prove Theorem \ref{thm:instability} we will first deduce from Theorem \ref{BDP_p1} an analogue of the capture lemma: if there is a saddle $q\in H(p, F)$ that can be turned into a sink by a small perturbation, then there is another saddle $Q\in H(p, F)$ that not only becomes a sink after an appropriate perturbation, but also catches a part of the unstable manifold of the continuation of $p$ into its basin of attraction. Since the proof of this lemma is a little technical, it is presented in a separate subsection.

Then a localized version of Theorem~\ref{thm:instability} can be proved similarly to Theorem~\ref{thm:NI}, the only difference is that we should obtain new sinks with the help of Theorem~\ref{BDP_p1} and Franks' lemma and use the new capture lemma instead of the old one.

The global version of Theorem~\ref{thm:instability} is obtained from the local version by an argument of the Kupka-Smale type similar to the proof of Cor. 0.3 in~\cite{BDP}.

\subsection{Another capture lemma}

\begin{lem}[another capture lemma]\label{lem:capture2}
For any $\e_0>0$ there exists $\e<\e_0$ such that the following holds.
Suppose that $H(p,F)$ does not admit any dominated splitting and the point $q$ provided by Theorem~\ref{BDP_p1} for a given $\e$ yields a contraction. Then by an $\e_0$-perturbation of $F$ in $\Diff^1(M)$ we can obtain a diffeomorphism $G$ such that the point $p$ is a hyperbolic saddle\footnote{and it is a hyperbolic continuation of the saddle $p$ of $F$.} for $G$ and $W^u(p,G)$ intersects the attraction basin of some sink. 
\end{lem}

\subsubsection*{Idea of the proof}
Consider the saddle $q$ given by Theorem~\ref{BDP_p1}. Since it is heteroclinically related to $p$, there is a transverse heteroclinic orbit $O(z)$ that accumulates to $O(p)$ in the past and to $O(q)$ in the future.
When we perturb the map in the neighborhood of $O(q)$ in order to turn $q$ into a sink, we cannot a priori be sure that $z$ will end up in the basin of this sink. It may happen, informally speaking, that the orbit of $z$ will leave the neighborhood where the map was perturbed just after a few iterations and long before $O(z)$ feels any attraction towards the new sink.

In order to circumvent this we will find another saddle $Q$ heteroclinically related to $p$ and such that the orbit of $Q$ makes $k$ winds close to the orbit of $q$, $k$ being arbitrarily large, and then closes after a few iterations. Denote by $x$ a heteroclinic point that goes from $O(p)$ to $O(Q)$. We will perform the perturbation in the neighborhood of $O(Q)$ in the following way. First, as in Franks lemma, the orbit of $Q$ itself will not be perturbed. Second, during the first $k_1$ winds there will be no perturbation, and we will let the images of $x$ approach the orbit of $Q$ as they do for the unperturbed map $F$. We need these images to come so close to $O(Q)$ that when we finally perturb the map in the neighborhood of the $(k_1+1)$-th wind, the corresponding points of the orbit of $x$ would stay in this neighborhood during the whole wind. Since the wind is close to $O(q)$, the perturbation can be so chosen that the image of $x$ is actually attracted to $O(Q)$ during the wind. Then we can repeat the same procedure during the rest of the winds. 

We will make sure that the number of these winds $k_2 = k - k_1$ is very large, much larger then $k_1$. Then $Q$ will become a sink no matter that there was no perturbation done during the first $k_1$ winds and there would be no perturbation during those few iterations that the orbit of $Q$ may spend far from the orbit of $q$. Moreover, if $k_2$ is large enough, the future orbit of $x$ will be attracted to this sink, which will conclude the proof, because $x\in W^u(O(p))$. 

\subsubsection*{Saddle $Q$}

Since $q\in H(p,F)$, there are transverse homoclinic intersections associated with $q$. Then there is also a non-trivial basic set $\Lambda\subset H(p,F)=H(q,F)$ such that $q\in\Lambda$ (see, for example, \cite[Thm 6.5.5]{Katok}). For a fixed $\delta>0$ and any $k\in\nn$, there is a periodic saddle $Q\in\Lambda$ with the following property: the orbit of $Q$ makes at least $k$ winds $\delta$-close to the orbit of $q$, and then closes after a number of iterations limited by a fixed integer that does not depend on~$k$. 

Indeed, any basic set admits Markov partitions of arbitrarily small diameter (see \cite[$\S 18.7$]{Katok}). Take a Markov partition of diameter less than $\delta$ for $\Lambda$ and consider the corresponding\footnote{See \cite[Thm 18.7.4.]{Katok}.} transitive Markov chain $(\Sigma_A, \sigma_A)$ that is semi-conjugated to the dynamics on $\Lambda$. Take a finite word $w$ defined by which rectangles of partition are consecutively visited by the orbit of $q$, then consider its $k$-th power (under concatenation) $[w]^k$ and take another finite word $w_0$ such that, first, $[w]^kw_0$ is not a power of any word and, second, our Markov chain admits a transition between the last letter of $w_0$ and the first letter of $w$. Note that $w_0$ can be taken independent of $k$, at least if we assume that there is a rectangle of the partition not visited by $O(q)$.  Then the periodic sequence $\omega\in \Sigma_A$ defined by the word $[w]^kw_0$ exists, and it is mapped to the required periodic point $Q$ by the semi-conjugacy map. 

Let us state explicitly that when we speak about the ``winds'' of the orbit $O(Q)$ around the orbit $O(q)$, we assume that $Q$ is $\delta$-close to $q$ and the same is true for $F^j(Q)$ and $F^j(q)$ for $j=1,2,\dots,kn$, where $n$ is the period of $q$. Each wind is a subset of $O(Q)$ that consists of $n$ consecutive points, namely the points $F^{n(j-1) + i}(Q), \; i = 0,\dots, n-1$, for the $j$-th wind. We will denote the period of $Q$ by $N$. Obviously, $N>kn$. 

Observe that by taking $\delta$ small enough we can ensure that the differentials at the points of $O(Q)$ that belong to the winds are $\e$-close to differentials at the corresponding points of $O(q)$. Therefore the composition of differentials along any wind can be turned into a contraction by $2\e$-perturbations of these differentials. Since the number of iterations that $O(Q)$ spends far from $O(q)$ is limited whereas $k$ may be taken arbitrarily large, we conclude that the composition of differentials along the whole $O(Q)$ may be turned into a linear contraction as well.

From now on we assume for simplicity, that for each point of the orbit $O(Q)$ we fix local coordinates with the origin at this point and whenever we consider $F$ restricted to a small neighborhood of the orbit $O(Q)$, we use these local coordinates. Then we may informally speak, for instance, about a linear mapping from the vicinity of $Q$ to the vicinity of $F(Q)$, or even about the mapping coinciding or $C^1$-close to $dF(Q)$. A formal way to say the same thing would be to consider an exponential mapping (given by the Riemannian metric) and then consider $\exp\circ\, dF(Q)\circ\exp^{-1}$. Moreover, we will switch to the euclidean metric and vector norm given by our fixed coordinates. When we change the vector norm, $\e$-perturbations of the differentials become $C\e$-perturbations for some positive constant $C$, so let us redefine $\e$ so that we do not need to write this constant every time.

The points $p$ and $Q$ are heteroclinically related (since this relation is transitive), therefore $W^u(O(p))$ transversally intersects $W^s(O(Q))$ at some point $x$. Replacing if necessary the points $p,Q,x$ by their images under some iteration of $F$ we can find a number $r<\min(\delta,\e)$ such that for the $r$-neighborhood $W$ of $O(Q)$ the following holds:

\begin{itemize}\setlength\itemsep{-0.1em}
\item{$W$ is a union of disjoint balls $B_j, \; j =1,\dots,N$ of radius $r$ centered at the points of $O(Q)$;}
\item{in each ball $B_j$ the map $F$  is $\e/10$-close in $C^1$-topology to the linear map that corresponds to the differential at the center of the ball, i.e., to $dF(F^{j-1}(Q))$;}
\item{$x\in W$, but the past semi-orbit of $x$ does not intersect $W$;}
\item{$x\in W^s_{loc, r}(Q)\cap W^u(p)$.}

\end{itemize}
In what follows we do not perturb $F$ outside $W$, therefore the point $x$ stays at the unstable manifold of the saddle $p.$
We will perform the perturbation of $F$ in $W$ in two steps.

\subsubsection*{Step 1: no perturbation during the first $k_1$ winds}
 
During the first step that consists of $k_1$ winds we do not perform any perturbations. The number $k_1$ should be chosen so large that the following inequality holds:
\begin{equation}\label{ineq:ratio}
\frac{\dist\left(F^{nk_1}(x), F^{nk_1}(Q)\right)}{\dist(x, Q)} < \frac{1}{10(L+1)^{n-1}},
\end{equation}
where $L$ is the Lipschitz constant for $F$.\footnote{Note that we can not argue like that in the case of an arbitrary saddle. Imagine a situation when there is a weak repulsion during nearly the whole periodic orbit and at the very end a few decisive iterations make it a saddle.} The distance between the images of $x$ and $Q$ decreases because $Q\in\Lambda$ and $x\in W^s_{loc, r}(Q)$. Indeed, it follows from the definition of a hyperbolic set that there are constants $c>0$ and $\lambda<1$ such that for any $z\in\Lambda$ and any $y\in W^s_{loc, r}(z)$, if $r$ is small enough, the following holds: 

$$\forall j\in\nn,\;\;\dist(F^j(y), F^j(z)) \le c\lambda^j\cdot\dist(y,z).$$

This inequality implies that (\ref{ineq:ratio}) holds for large values of $k_1$. 

Denote $x_1 = F^{nk_1}(x)$ and $Q_1 = F^{nk_1}(Q)$. Since $\dist(x, Q) \le r$, inequality~(\ref{ineq:ratio}) implies that $\dist(x_1, Q_1) < r/(10(L+1)^{n-1})$. 

\subsubsection*{Step 2: actual perturbation}

Now the second step begins. It consists of $k_2=k-k_1$ winds. 

Theorem \ref{BDP_p1} provides $n$ linear maps $A_1,\dots, A_n$ such that these are $\e$-perturbations of the differential along the orbit of $q$ and the composition $A = A_n\circ\cdots\circ A_1$ is a contraction. Let us assume for simplicity that $A$ contracts the euclidean norm. In the general case this is true for some power of $A$ and the argument should be modified accordingly. 

Consider the $(k_1+1)$-th wind of $O(Q)$ around $O(q)$ that starts at $Q_1$ and continues up to $Q_n := F^{n-1}(Q_1)$.  
Recall that, since $Q_1$ is $\delta$-close to $F^{nk_1}(q)=q$, we assume that the map $A_1$ is a $2\e$-perturbation of the differential at $Q_1$, and, analogously, each $A_j$ is a $2\e$-perturbation of the differential at the corresponding point of the wind.

Take a ball $B(Q_1, \frac{r}{10})$ of radius $\frac{r}{10}$ centered at $Q_1$ and a ten times larger ball $B(Q_1, r)$ with the same center  (this larger ball actually coincides with the previously defined ball $B_j$ with $j = nk_1+1$). 
We can modify $F$ inside $B(Q_1, r)$ in such a way that inside $B(Q_1, \frac{r}{10})$ the resulting map would coincide with the map $A_1$. This can be made by a $(c_1\cdot 2\e)$-perturbation, where $c_1\ge 1$ depends on the radii of the two balls, but not on $\e$.\footnote{Actually, making the radii small and their ratio large, we could take $c_1$ arbitrarily close to~1.} 

Analogously, for any $j\in\{2,\dots, n\}$ we can take balls of radii $\frac{r}{10}$ and $r$ centered at $Q_j := F^{j-1}(Q_1)$ and modify $F$ inside the bigger ball so that the restriction of the new map to the small ball coincide with $A_j$. Since for different $j$ the big balls do not intersect, it still takes merely a $(c_1\cdot 2\e)$-perturbation to perform the overall modification. 

We will preserve the notation $F$ for the modified map. It is important that after this modification the point $x$, in general, does not belong to the stable manifold of $Q$. However, at the end of the wind the corresponding iterate of $x$ is still inside the ball $B(Q_n, r)$.

Indeed, recall that we have denoted by $d$ the distance between $Q_1$ and $x_1\in O(x)$.  Obviously, the distance between $x_2 = F(x_1)$ and $Q_2=F(Q_1)$ is less than $d\cdot (L+1)$, where $L$ is the Lipschitz constant for the original $F$: we assume that $\e$ is small and add 1 to $L$ in order to take the perturbation into account. Likewise, we have for $x_j = F^{j-1}(x_1)$ that $\dist(x_j, Q_j) < d\cdot (L+1)^{j-1}$.
Since, as we required during the first step, $d\cdot 10\cdot (L+1)^{n-1} < r$, we conclude that during this wind the points of $O(x)$ stay inside the union of smaller balls where the original map was replaced by the maps $A_1,\dots, A_n$.
 
Recall that we assume the composition $A = A_n\circ\cdots\circ A_1$ to be a contraction map in our euclidean metric. Denoting by $\lambda_1$ the minimal rate of this contraction we obtain 
$$\dist(F^{n}(x_1),F^{n}(Q_1)) \le \lambda_1\cdot d < d.$$
This means that we can repeat the same modification procedure for the next wind and further on up to the end of the $k$-th wind.

\subsubsection*{Attraction to the sink}

 If $k_2$ is large enough, after Step 2 the point $Q$ becomes a sink and $x$ is in its basin of attraction. Indeed, let us show that if $k_2$ is sufficiently large, then any point $y$ that is $d$-close to $Q_1$ (in particular, the point $x_1\in O(x)$) is attracted to this sink. Recall that $N = {\rm per \,} (Q) = (k_1+k_2)n + r_0$, where $r_0 = \mathrm{length}(w_0)$ is smaller than some constant whereas $k_1$ and $k_2$ may be chosen arbitrary large, and we have already chosen $k_1$. We have 
$$\dist{(F^{N}(y),Q_1)} = \dist{(F^{N}(y),F^N(Q_1))}\le \left(\lambda_1^{k_2}\cdot(L+1)^{k_1n+r_0}\right)\cdot\dist{(y, Q_1)}.$$
For a large $k_2$ the inequality $\lambda_1^{k_2}\cdot(L+1)^{k_1n+r_0} < \frac{1}{2}$ holds. Then we have 
$$\dist{(F^{N}(y),Q_1)} \le \frac{1}{2}\dist{(y, Q_1)}.$$ 
This shows that $x_1$ and its preimage $x$ are both attracted to the sink. 

Thus after a $2c_1\e$-modification of the initial map inside the domain $W$ we see that:
\begin{itemize}\setlength\itemsep{-0.1em}
\item{the past semi-orbit of $x$ is unchanged, and consequently $x$ is still in $W^u(p)$;}
\item{the orbit of $Q$ is unchanged, but $Q$ becomes a sink, and $x$ is in its basin of attraction.}
\end{itemize}
If $\e$ is small enough, we have $2c_1\e < \e_0$. Then we have obtained an $\e_0$-modification of the initial $F$ with the required property. This completes the proof of the lemma. 

\subsection{Local version of Theorem~\ref{thm:instability}}

The following statement is a localized version of Theorem \ref{thm:instability}.

\begin{thm}  \label{prop:localized}
Suppose that $F\in\Diff^1(M)$ has a neighborhood $U$ where the hyperbolic continuation of a periodic saddle $p$ of $F$ is defined, and, moreover, for any $G\in U$ this saddle is dissipative.
Then for a Baire-generic $G\in U$ either $H(p(G),G)$ admits a dominated splitting, or $A_M(G)$ is Lyapunov unstable. 
\end{thm}

\begin{proof}
We will show, for a generic diffeomorphism $G\in U$, that if $H(p(G),G)$ does not admit a dominated splitting, then $G$ satisfies both assumptions of Proposition \ref{prop:suff}, i.e., there is a sequence of sinks accumulating to $p(G)$ and the unstable manifold of $p(G)$ intersects the basin of some sink. Then Proposition \ref{prop:suff} will imply that the Milnor attractor $A_M(G)$ is unstable.
 
Denote by $DS(U)$ the subset of $U$ that consists of diffeomorphisms $G$ for which $H(p(G),G)$ admits a dominated splitting.
Consider the interior of $DS(U)$ and denote by $V$ the complement of the closure of this interior in $U$: $V = U\setminus {\rm Cl}({\rm Int} DS(U))$. It follows from the definition that $V$ contains a dense subset of diffeomorphisms $G$ for which $H(p(G),G)$ does not admit any dominated splitting. If $V$ is empty, we are done: the homoclinic class $H(p(G),G)$ admits a dominated splitting for a topologically generic $G\in U$.

If $V\ne\emptyset$, Theorem \ref{BDP_p1} (combined with Franks' lemma) implies that any diffeomorphism $G\in V$ can be approximated by a diffeomorphism with a sink or a source $s$ close to the continuation of $p$. Since we assume that $p(G)$ is dissipative, by Remark \ref{rem:sinks} we can also assume that $s$ is a sink.

Then we can use the Newhouse argument as in the proof of Theorem \ref{thm:NI} (see Section \ref{sec:thmNI}). The only difference is that new sinks are obtained not by unfolding homoclinic tangencies but with the help of Theorem \ref{BDP_p1} as above. This argument yields that generically in $V$ sinks accumulate to the hyperbolic continuation of $p$. 

Further note that Lemma \ref{lem:capture2} can be applied to any diffeomorphism $G\in V$ for which $H(p(G),G)$ does not admit a dominated splitting, and such diffeomorphisms are dense in $V$. Then there is an open and dense subset of $V$ where for any diffeomorphism $G$ the unstable manifold $W^u(p(G), G)$ intersects a basin of a sink. Thus, there is a residual subset $R$ of $V$ where both assumptions of Proposition \ref{prop:suff} are satisfied and therefore $A_M$ is unstable.

The observation that the union $R\cup DS(U)$ is a residual subset of $U$ completes the proof.   
\end{proof}

\subsection{Global version of Theorem~\ref{thm:instability}}

Theorem \ref{thm:instability} may be proved now by essentially the same argument as Cor. 0.3. in~\cite{BDP}.

\begin{proof}[Proof of Theorem \ref{thm:instability}]
 Diffeomorphisms for which all periodic points of period less than $n$ are hyperbolic form an open and dense subset $U_n$ of $\Diff^1(M)$. Let us split $U_n$ into the union of open subsets $U_{n,\a}$ such that for $F\in U_{n,\a}$ the number of saddles of period less than $n$ is constant and equal to $k(\a)$, and these saddles vary continuously with the map. Take some $U_{n,\a}$ and denote those saddles by $p_1,\dots,p_k$. 

For each $j$ consider a set $DS(p_j)\subset U_{n,\a}$, where $H(p_j(G), G)$ admits a dominated splitting. Then fix $j$ and consider an open set $V_j = U_{n,\a}\setminus {\rm Cl}({\rm Int\,} DS(p_j))$. Denote by  $V^+_j$ (resp. $V^-_j$) an open subset of $V_j$ that consists of diffeomorphisms for which $p_j$ is dissipative (resp. area-expanding). The union $V^+_j\cup V^-_j$ is dense in $V_j$. Application of Theorem \ref{prop:localized} to $V^+_j$ yields a residual subset $R^+_j\subset V^+_j$ where diffeomorphisms have unstable Milnor attractors. An analogous argument for $V^-_j$ in the inversed time provides a residual $R^-_j\subset V^-_j$ such that for each $F$ in this set the inverse map $F^{-1}$ has an unstable Milnor attractor. Then $R_j=R^-_j\cup R^+_j$ is residual in $V_j$. The union of $R_j$ and $DS(p_j)$ is a residual subset of $U_{n,\a}$. Intersecting these $R_j\cup DS(p_j)$ we obtain a residual subset $R_{n,\a}$ of $U_{n,\a}$. For any $F\in R_{n,\a}$ either homoclinic classes of all saddles $p_j$ admit some dominated splittings or the Milnor attractor of $F$ or $F^{-1}$ is unstable. Finally, $R=\bigcap\limits_{n}\bigcup\limits_\a R_{n,\a}$ is a global residual subset such that for any $F\in R$ either every homoclinic class admits a dominated splitting, or $A_M(F)$ is unstable for $F$, or $A_M(F^{-1})$ is unstable for~$F^{-1}.$   
\end{proof}

\vspace{0.9cm}

Ivan Shilin,

Moscow State University

\vspace{0.1cm}

E-mail: i.s.shilin@yandex.ru


\begin{thebibliography}{100}



\bibitem [A]{Masa} M. Asaoka: Hyperbolic sets exhibiting $C^1$-persistent homoclinic tangency for higher dimensions, Proc. Amer. Math. Soc., \textbf{136}, no. 2 (2008), 667--686.

\bibitem [AAIS]{AAIS} V. I. Arnold, V. S. Afraimovich, Yu. S. Ilyashenko, L. P. Shilnikov: Bifurcation theory. Dynamical Systems v. 5 (Moscow 1986). Springer, 1994.

\bibitem [BDP]{BDP} C. Bonatti, L. D{\'i}az, E. R. Pujals: A $C^1$-generic dichotomy for diffeomorphisms: weak forms of hyperbolicity or infinitely many sinks or sources, Ann. of Math. \textbf{158} (2003), 355--418.

\bibitem [BLY]{BLY} C. Bonatti, M. Li, D. Yang: On existence of attractors, Trans. Amer. Math. Soc. \textbf{365} (2013), 1369--1391.

\bibitem [C]{Crovisier14} S. Crovisier: Dynamics of $C^1$-diffeomorphisms: global description and prospects for classification, preprint, 2014.

\bibitem [F]{Franks} J. Franks, Necessary conditions for the stability of diffeomorphisms, Trans. A. M. S., \textbf{158} (1971), 301--308.

\bibitem [GI]{GI} A. Gorodetski, Yu. Ilyashenko: Minimal and strange attractors. Internat. J. Bifur. Chaos {\bf 6} (6) (1996), 1177--1183.

\bibitem [GH]{GUCK} J. Guckenheimer, P. Holmes: Nonlinear oscillations, dynamical systems and bifurcations of vector fields. Berlin, Heidelberg, New York: Springer, 1985.

\bibitem [Il91]{Il91} Yu. Ilyashenko: The concept of minimal attractors and maximal attractors of partial differential equations of the Kuramoto–Sivashinski type, Chaos {\bf 1} (2) (1991), 168--173.

\bibitem [Il03]{Il03} Yu. Ilyashenko: Minimal Attractors, EQUADIFF 2003, 421-428, World Sci. Publ. Hackensack, NJ, 2005.

\bibitem [KH]{Katok} A. Katok B. Hasselblatt, Introduction to the modern theory of dynamical systems, Cambridge University Press, Cambridge, 1997.

\bibitem [M]{Milnor} J. Milnor: On the concept of attractor, Comm. Math. Phys. \textbf{99} (1985), 177--195.

\bibitem [N1]{N70} S. Newhouse: Non-density of Axiom A(a) on $S^2$, Global Analysis, Proc. Symp. in Pure Math., Volume XIV, A.M.S., pp.191--203, 1970.

\bibitem [N2]{N74} S. Newhouse: Diffeomorphisms with infinitely many sinks, Topology, \textbf{13} (1974), 9--18.

\bibitem [N3]{N79} S. Newhouse: The abundance of wild hyperbolic sets and nonsmooth stable sets for diffeomorphisms, Inst. Hautes Etudes Sci. Publ. Math., \textbf{50} (1979), 101--151.

\bibitem [N4]{CIME} S. Newhouse: Lectures on dynamical systems, Dynamical Systems, Springer Berlin Heidelberg 2010 (Reprint of the 1st ed. C.I.M.E., Ed. Liguori, Napoli \& Birkh{\"a}user 1980), 209--312.

\bibitem [N5]{N_new} S. Newhouse: New phenomena associated with homoclinic tangencies, Ergod. Th. Dynam. Systems, \textbf{24}(5) (2004), 1725--1738.

\bibitem [PV]{PV} J. Palis, M. Viana: High dimension diffeomorphisms displaying infinitely many periodic attractors, Ann. of Math., \textbf{140} (1994), 207--250. 

\bibitem [PT]{PT} J. Palis, F. Takens: Hyperbolicity and sensitive chaotic dynamics at homoclinic bifurcations, Cambridge Studies in Advanced Mathematics, \textbf{35}. Cambridge University Press, Cambridge, 1993.

\bibitem [Ply]{Ply} R.V. Plykin: Sources and sinks of A-diffeomorphisms of surfaces, Math. of the USSR-Sbornik, 23:2 (1974), 233--253.

\bibitem [TS]{TS} J.C. Tatjer, C. Sim{\'o}: Basins of attraction near homoclinic tangencies, Ergod. Th. Dynam. Systems, \textbf{14} (1994), 351--390. 





\end{thebibliography}
\end{document}